\documentclass{amsart}

\usepackage{url}
\usepackage{color}
\definecolor{darkgreen}{rgb}{0,0.5,0}
\usepackage[
        colorlinks, citecolor=darkgreen,
        backref,
        pdfauthor={J. Steffen Mueller, Berno Reitsma},
        pdftitle={Computing torsion subgroups of Jacobians of hyperelliptic curves of genus~3}
]{hyperref}

\usepackage{url}

\usepackage{tikz,enumerate,caption,stmaryrd,amsfonts,amssymb,systeme,comment,graphicx,colonequals,faktor,xfrac}
\usepackage{graphicx}
\usepackage{amssymb,amsmath, amsthm}
\usepackage{comment}
\usepackage{algorithmic,algorithm}
\usepackage[OT2,OT1]{fontenc}
\usepackage{mathtools}
\usepackage{colonequals}
\usepackage{epsfig}
\usepackage{fancyhdr}
\usepackage[shortlabels]{enumitem}

\usepackage[all]{xy}
\usepackage{tikzcd}

\newcommand\cyr{%
  \renewcommand\rmdefault{cmr}%
  \renewcommand\sfdefault{wncyss}%
  \renewcommand\encodingdefault{OT2}%
  \normalfont\selectfont}
\DeclareTextFontCommand{\textcyr}{\cyr}

\definecolor{red}{rgb}{0.9,0,0}
\definecolor{purple}{rgb}{0.8,0,0.6}

\numberwithin{equation}{section}

\newtheorem{thm}{Theorem}[section]

\newtheorem{prop}[thm]{Proposition}
\newtheorem{ass}[thm]{Assumption}

\newtheorem{lemma}[thm]{Lemma}
\newtheorem{cor}[thm]{Corollary}

\theoremstyle{definition}

\newtheorem{defn}[thm]{Definition}

\theoremstyle{remark}

\newtheorem{algo}[thm]{Algorithm}
\newtheorem{rk}[thm]{Remark}
\newtheorem{ex}[thm]{Example}

\newcommand\Q{\mathbb{Q}}

\newcommand\F{\mathbb{F}}

\newcommand\A{\mathbb{A}}
\newcommand\Z{\mathbb{Z}}
\newcommand\R{\mathbb{R}}

\newcommand\Gal{\mathop{\rm Gal}\nolimits}
\newcommand\Sym{\mathop{\rm Sym}\nolimits}

\newcommand\supp{\mathop{\rm supp}\nolimits}

\newcommand\tors{\mathop{\rm tors}\nolimits}

\newcommand{\Div}{\operatorname{Div}}
\newcommand{\Pic}{\operatorname{Pic}}

\renewcommand{\div}{\operatorname{div}}

\newcommand{\BP}{{\mathbb P}}
\newcommand{\eps}{\varepsilon}

\newcommand{\JQt}{J(\mathbb{Q})_{\tors}}
\newcommand{\AQt}{A(\mathbb{Q})_{\tors}}

\begin{document}

\title[Torsion subgroups of hyperelliptic genus-3 Jacobians]{Computing torsion subgroups of Jacobians of hyperelliptic curves of genus~3}

\author{J. Steffen M\"uller}
\author{Berno Reitsma}
\address{
  Bernoulli Institute, 
  University of Groningen,
  Nijenborgh 9,
  9747 AG Groningen,
  The Netherlands
}

\date{\today}


\begin{abstract} \setlength{\parskip}{1ex} \setlength{\parindent}{0mm}
We introduce an algorithm to compute the structure of the rational torsion subgroup of the Jacobian of a
  hyperelliptic curve of genus~3 over the rationals. We apply a {\tt Magma} implementation
  of our algorithm to a database of curves with low discriminant due to Sutherland as well as a
  list of curves with small coefficients. In the process, we find several torsion
  structures not previously described in the literature.
The algorithm is a generalisation of an algorithm for genus 2 due to Stoll, which we
  extend to  
  abelian varieties satisfying certain conditions. The idea is to compute
  $p$-adic torsion lifts of points
  over finite fields using the Kummer variety and to check whether
  they are rational using  heights. Both  have been made explicit for
  Jacobians of hyperelliptic curves of genus~3 by Stoll.
This article is partially based on the second-named author's Master thesis.
\end{abstract}

\maketitle


\section{Introduction}\label{intro}
For an abelian variety $A/\Q$,
the torsion subgroup 
$A(\mathbb{Q})_{\tors}$ of the group $A(\Q)$ of $\Q$-rational points on $A$ is finite.
If $A=E$ is an elliptic curve, it is easy to compute
$E(\mathbb{Q})_{\tors}$, 
and for Jacobians of genus~2 curves, there is a $p$-adic algorithm due to
Stoll (see~\cite[Section~11]{StollG2}). In the present paper, we give a theoretical
extension of Stoll's algorithm to arbitrary abelian varieties $A/\Q$. We then make this
extension practical for Jacobians of hyperelliptic curves of genus~3. The latter heavily
uses explicit arithmetic on the Kummer variety of such a Jacobian, also due to
Stoll~\cite{StollG3}.

Our main motivation comes from a database of hyperelliptic curves of genus~3 
due to Andrew Sutherland~\cite{sutherlanddatabase}. Similar to databases of elliptic
curves and curves of genus~2 in the {\tt LMFDB}~\cite{LMFDB}, it would be
useful to compute the most important arithmetic invariants
of these curves, including the structure of the subgroup of rational
torsion points on its Jacobian.
Sutherland asked for an algorithm to accomplish this in~2017.
We have used our algorithm to compute the torsion subgroups of all curves
in the
database, see~\S\ref{subsec:suth}.

In this computation we found several torsion structures that were not previously known in
the literature. Recall that for elliptic curves over $\Q$, Mazur's Theorem gives a complete list of all  
torsion subgroups up to isomorphism. For 
dimension $d>1$, it is not
even known whether there is a uniform bound on the size of all 
rational torsion subgroups of abelian varieties over $\Q$ of dimension $d$. A lot of work
has gone into constructing Jacobians of genus~2 curves with large torsion orders (see for
instance~\cite{How15} and the references therein).
Some constructions of rational torsion points of large order on Jacobians of hyperelliptic genus~3 curves can be
found in~\cite{kronbergPhD}, \cite{JacobianDescent}, \cite{SplitJacobians} and
in~\cite{Fly91, Lep97}, where families of Jacobians with large rational torsion are
constructed that contain hyperelliptic genus~3 examples. A list of orders of rational
torsion points for such curves known in the literature can be found in~\cite[Table~3.2]{JacobianDescent}. However,
much less is known than for genus~2. Therefore it is
interesting to investigate which abelian groups actually occur. 
Inspired by a search by Howe for $g=2$~\cite{How15}, we ran through a list of certain
hyperelliptic genus~3 curves
with small coefficients, and we found many new torsion structures in this
way, see~\S\ref{subsec:search}.
We obtain
the following list of all torsion structures that are currently known to
occur.
\begin{thm}\label{T:main}
Every abelian group of order $\le 44$ is isomorphic to 
the group of rational torsion
points on 
  a geometrically simple Jacobian of  a hyperelliptic curve over $\Q$ of genus~3, {with
    the possible exception of} the groups
with invariant factors
$
    [3,3,3], [3,9], [2,4,4], [6,6].
$
In addition, the abelian groups with the following invariant factors are isomorphic to 
the group of rational torsion
points on a geometrically simple Jacobian of a hyperelliptic curve over $\Q$ of genus~3:
    \begin{align*}
      & [46], [2,2,2,6], [2,2,12], [2,24], [4,12], 
      [48], [49], [50], [51], [2,26], [52], [3,18], [54], [2,2,14], [2,28], [56],
      [58],\\&[2,30], [63],[2,2,2,2,2,2],[2,2,2,2,4], [2,2,2,8],[2,4,8], [2,32], [64], [65], [70], [6,12], [72],
      [2,2,2,10], [2,2,20],\\& [2,42], [2,44], [91],[2,2,28],[2,52],
 [2,2,2,2,10], 
    \end{align*}
\end{thm}
All torsion structures in Theorem~\ref{T:main} came up in our search or in Sutherland's
database, except for the groups {$(\Z/2\Z)^5,$ $(\Z/2\Z)^6, (\Z/2\Z)^4\times \Z/4\Z$
and $(\Z/2\Z)^3\times \Z/6\Z$, which we constructed.
We do not claim that the groups listed as exceptions in
Theorem~\ref{T:main} do not occur; we simply did not find such examples in our computations
or the literature.}
Using our computations we found examples for all torsion structures that
appeared in the literature prior to our work; 
in particular, we found new examples for the largest known prime group order~43
and the largest known
point order~91, both exhibited by Nicholls~\cite{JacobianDescent}. The group
$(\Z/2\Z)^4\times \Z/10\Z$ is the largest group of rational torsion points
on a geometrically simple Jacobian
of a hyperelliptic genus~3 curve found so far; no such group of size $>91$ was previously known.

\begin{rk}\label{R:}
  We focused on geometrically simple examples.
  More generally, we have found, for every abelian group $A$ of order $<45$
  except for the groups with invariant factors $[3,3,3]$ and $[3,9]$, 
  a Jacobian of a hyperelliptic curve over $\Q$ of genus~3 with group of
  rational torsion points isomorphic to $A$.
  We expect that many additional structures
  can be found by systematically gluing abelian varieties of lower dimension, for instance using the methods of~\cite{HSS20}. 
\end{rk}

There are other possible applications of our algorithm: The order of the rational torsion subgroup appears in the strong version of
the conjecture of Birch and Swinnerton-Dyer, and we therefore need an algorithm to compute
this quantity to gather empirical evidence for the conjecture. Finally, if
$J$ is the Jacobian of a smooth projective curve $X/\Q$ with $\mathrm{rk}J(\Q)=0$, and we have an Abel-Jacobi embedding $j\colon X\to J$ defined over $\Q$, then we
can compute the set $X(\Q)$ by finding $J(\Q) = \JQt$ and checking which
points $P\in \JQt$
have a rational preimage under $j$.

\subsection{Upper bounds using reduction}\label{subsec:red}
Let $A/\Q$ be an abelian variety.
An upper bound on the order of $\AQt$ can be computed easily as follows: 
For a prime $p$ of good reduction for $A$ and an integer $m$ (which we
require to be odd if $p=2$), the restriction of the reduction map 
  \begin{equation*}
  \rho_p\colon
    A(\Q_p)\to \tilde{A}(\F_p)
  \end{equation*}
to $A(\mathbb{Q}_p)[m]$ is injective, where $\tilde{A}/\F_p$ is the reduction of $A$
modulo $p$ (see \cite[Theorem C.1.4]{DiophantineGeometry}). We choose a set
$S$ containing a few small odd primes of good reduction and
compute $\#\tilde{A}(\F_p)$ for all $p \in S$; then 
\begin{equation*}
\#\AQt\mid \gcd_{p\in S} \#\tilde{A}(\F_p)\,.
\end{equation*}

We can obtain more information from the structure of $\tilde{A}(\F_p)$
rather than only its order. 
\begin{ex}\label{extors3}
Consider the Jacobian $J$ of
\[
  X \colon y^2 = x^8 + 2x^7 + 3x^6 + 4x^5 + 9x^4 + 8x^3 + 7x^2 + 2x + 1\equalscolon f(x)\,.
\]
The primes of bad reduction for
$X$ are $2$, $3$ and $13177$. We find
$\#\tilde{J}(\mathbb{F}_5) = 180$, $\#\tilde{J}(\mathbb{F}_7) = 666$, so that $\#\JQt\mid
18$. A closer inspection shows
  \[
\tilde{J}(\mathbb{F}_5) \cong \mathbb{Z}/3\mathbb{Z} \times
  \mathbb{Z}/60\mathbb{Z}\,;\quad 
\tilde{J}(\mathbb{F}_7) \cong \mathbb{Z}/666\mathbb{Z}\,.
\]
  We conclude that $J(\mathbb{Q})_{\tors}$
is isomorphic to a subgroup of $\mathbb{Z}/6\mathbb{Z}$. 
  We will see in Example~\ref{2torsextors3} that $\#J(\Q)[2] =2$. 
To find $\#\JQt$, it remains to check whether there is a rational
  point of order~3. Searching among small rational points on $X$, we find that 
$[(0,-1) - \infty_1]$ has this property, where $\infty_1$ is the point with coordinates
  $(0,1)$ on the model 
  $$w^2 = 1+2z +3z^2+4z^3+9z^4+8z^5+7z^6+2z^7+z^8\,.$$
\end{ex}
Most of the time, the upper bound that we get from considering the structure of
$\tilde{A}(\F_p)$ for a reasonable number of primes $p$ of good reduction
is actually equal to the correct order. For
instance, in the database~\cite{sutherlanddatabase}, we found this to be the case 
for more than  $97\%$ of all Jacobians, where we used all good primes
below~1000. 
For the remaining
ones, the quotient is a small power of~2 in the vast majority
of cases.  See~\S\ref{subsec:suth} for more details. 

Example~\ref{extors3} has the convenient property that $X$ has a rational point, which allows us to
add points in $J(\Q)$. The computer algebra system {\tt Magma}~\cite{BCP97} 
contains an algorithm to compute the group law in $J(k)$ for the Jacobian of a
hyperelliptic curve of odd genus over a field $k$ if a $k$-rational point
on the curve is known; alternatively, one may use
   Sutherland's (more efficient) balanced divisor approach~\cite{sutherland}.

Now consider the following example, brought to our attention by Andrew Sutherland.
\begin{ex}\label{sutherland13}
Let
$X/\Q$ be the hyperelliptic curve defined by
\[
y^2 = 5x^8 - 14x^7 + 33x^6 - 36x^5 + 30x^4 + 2x^3 - 16x^2 + 20x - 7.
\]
with Jacobian $J/\Q$.
  There seems to be a point of order~13 in $\tilde{J}(\F_p)$ for all good primes $p$. Is
  there a global point of order~13? 
  The curve $X$ does not seem to have any rational points, so 
  arithmetic in
  $J(\Q)$ is not implemented. In any case, there are no obvious nontrivial points in $J(\Q)$.
  We will show in  Example~\ref{newrationalpoint} that $\JQt\cong \Z/13\Z$.
\end{ex}

Our method for computing $J(\Q)_{\mathrm{tors}}$ follows an approach due to
Stoll for dimension~2~\cite[Section~11]{StollG2}, and works as follows: We lift points of order
$m$ coprime to $p$ to $A(\Q_p)[m]$ and then check whether the lift is rational. To do so,
one potential approach is to represent points in $A(\Q_p)$ using a projective embedding
of $A$.
This, however, is much too complicated in practice, since in general one would have to work
in $\mathbb{P}^{4^g-1}$ and no explicit projective embedding is known for $g>2$.
Instead, {we follow Stoll in using} the Kummer variety of $A$. {This is practical
for Jacobians of hyperelliptic curves of genus~3, since the required explicit theory of
the Kummer variety and of heights was developed by Stoll in~\cite{StollG3}.}

\begin{rk}\label{R:raymond}
  In recent work~\cite{Bom23}, van Bommel has given an algorithm to compute the torsion
  subgroup for Jacobians of non-hyperelliptic curves of genus~3. His
  algorithm does not use the Kummer variety or height bounds.
\end{rk}

\subsection{Outline}\label{subsec:outline}
We gather
preliminaries on Kummer varieties and heights on abelian varieties in
Section~\ref{S:kum-heights}. In Section~\ref{algorithm} we generalise Stoll's algorithm
for the computation of $\JQt$ when $J$ is the Jacobian of a genus 2 curve to abelian
varieties $A/\Q$ that satisfy Assumption~\ref{assump}.
Then we show that this assumption is satisfied for Jacobians of hyperelliptic curves of
genus~3 in Section~\ref{genus3}. Finally, we discuss our computations  in
Section~\ref{examples}.

\subsection{Acknowledgements}\label{subsec:ack}
It is a pleasure to thank Andrew Sutherland for providing the motivation for this work and for helpful
discussions, and Michael Stoll for answering many questions and for useful
suggestions, in particular Lemma~\ref{L:deg8tors2}.
We thank Ludwig F\"urst, Timo Keller and especially Michael Stoll for many comments on 
preliminary versions of this article,  Max Kronberg for explaining results
from his thesis, and
Jaap Top and P\i{}nar K\i{}l\i{}\c{c}er for helpful discussions.
We would also like to thank two anonymous referees for careful reports with
many useful suggestions for improvement.
We are grateful to the Artificial Intelligence Group at the Bernoulli
Institute of the University of Groningen for providing access to the {\tt
Pallas}-server, which we used for our computations. 
The first author was supported by NWO Grant VI.Vidi.192.106.

\section{Kummer varieties and heights}\label{S:kum-heights}
If $A/k$ is an abelian variety of dimension $g>0$ over a field $k$, then the
\emph{Kummer variety} $K/k$ of $A$ is defined as $K \colonequals
A/\langle-1\rangle$. 
The quotient map is $2:1$ except at points of order $2$ in $A$, where it is
        injective. The images of these points are the singular points of $K$.
        By~\cite[Theorem~4.8.1]{CAV}, 
        $K$ can be embedded into $\mathbb{P}^{2^g-1}$.
	We fix a rational map
        \begin{equation}\label{kappa}
		\kappa \colon A \rightarrow \mathbb{P}^{2^g-1} 
        \end{equation}
        such that the image $\kappa(A)$ is a birational model for $K$.

Since $\kappa$ identifies inverses on $A$, the group structure
is lost, but scalar multiplication $[n]\colon A\to A$ descends, since it commutes with
inversion. In fact, there is a rational map $[[n]]\colon K\to K$ such that 
\begin{equation*}
\centering
\begin{tikzcd}
A \arrow[r, "\lbrack n \rbrack"] \arrow[d, "\kappa"]& A \arrow[d, "\kappa"] \\
K \arrow[r, "\lbrack \lbrack n \rbrack \rbrack"]& K
\end{tikzcd}
\end{equation*}
commutes. Furthermore, there is a rational map $B\colon \Sym^2(K)\to \Sym^2(K)$ which, for
$Q_1,Q_2\in A$, sends the unordered pair $\{\kappa(Q_1),\kappa(Q_2)\}$
to the unordered pair $\{\kappa(Q_1 + Q_2), \kappa(Q_1 - Q_2)\}$.
 We suppose that algorithms for the following tasks are available:
 \begin{itemize}
  \item \texttt{Double}: Given $\kappa(Q)$ for $Q \in A(k)$, return 
    $[[2]](\kappa(Q)) = \kappa(2Q)$.
\item \texttt{PseudoAdd}: Given $\kappa(Q_1), \kappa(Q_2), \kappa(Q_1 - Q_2)$ for
  $Q_1, Q_2 \in A(k)$, return $\kappa(Q_1 + Q_2)$.
\end{itemize}
This leads to the following double-and-add algorithm to compute $[[n]](R)$
for $n\in \Z\setminus\{0\}$ and $R\in K$.
\begin{algo}\label{kummerarithmetic} \textbf{Multiplication-by-$n$ on the Kummer}\\
  Input: $R \in K(k)$, $n \in \mathbb{Z}$\\
  Output: $[[n]](R)$
  \begin{enumerate}
  \item Set $\boldsymbol{x} \colonequals \kappa(0)$, $\boldsymbol{y}
    \colonequals R, \boldsymbol{z} \colonequals R$ and $m \colonequals |n|$.
\item While $m \neq 0$, repeat the following steps.
  \begin{enumerate}
    \item If $m$ is odd, then set $\boldsymbol{x} \colonequals \texttt{PseudoAdd}(\boldsymbol{x}, \boldsymbol{z}, \boldsymbol{y}).$
    Else, set $\boldsymbol{y} \colonequals \texttt{PseudoAdd}(\boldsymbol{y}, \boldsymbol{z}, \boldsymbol{x})$.
    \item Set $\boldsymbol{z} \colonequals \texttt{Double}(\boldsymbol{z})$.
    \item Set $m \colonequals \left\lfloor\frac{m}{2}\right\rfloor$.
	\end{enumerate}
      \item Return $\boldsymbol{x} \colonequals [[m]](R).$
\end{enumerate}
\end{algo}
Algorithm~\ref{kummerarithmetic} is a generalisation of 
the Montgomery ladder for elliptic curves; the genus 2 case is discussed
in~\cite{FlynnSmart}.

Now suppose that $k=\Q$. Then we can use the map
$\kappa$  to define heights on
$A(\Q)$ as follows.
The \emph{naive height} $h\colon A(\Q)\to \R_{\ge 0}$ is the function $h\colonequals \log
(H\circ \kappa)$, where $H \colon \mathbb{P}^{2^g-1}(\Q) \rightarrow \mathbb{R}_{\geq 0}$ is the
usual height given by 
 mapping $P= (x_1:\ldots:x_{2^g})$ to $\max(|x_1|, \ldots, |x_{2^g}|)$, where
 $x_1,\ldots,x_{2^g}$ are coprime integers.
 The map $h$ is quadratic up to a bounded function, hence the
\emph{canonical height}
is well-defined:
\[
\hat{h}(Q) \colonequals \lim_{n \rightarrow \infty} \frac{h(nQ)}{n^2}
\]
\begin{thm}\label{nerontate}
  \textbf{(Néron–Tate)} {\cite[Theorem~B.5.1]{DiophantineGeometry}}
The following properties are satisfied.
\begin{enumerate}
  \item $\hat{h}(nQ) = n^2\hat{h}(Q)$ for all $n \in \mathbb{Z}$ and $Q\in A(\Q)$.
  \item  For $Q\in A(\Q)$, we have $\hat{h}(Q) = 0$ if and only if $Q \in A(\mathbb{Q})_{\tors}$.
	\item The set $\{Q \in A(\mathbb{Q}) : \hat{h}(Q) \leq B\}$ is
          finite for every constant $B \geq 0$.
	\item The height difference $|\hat{h} - h|$ is bounded.
\end{enumerate}
\end{thm}
By Theorem~\ref{nerontate}(2), torsion points have small naive height. More
precisely, suppose that $\beta \in \R_{\ge 0}$ satisfies
\begin{equation*}
|\hat{h}(Q) - h(Q)| < \beta
\end{equation*}
for all $Q \in A(\mathbb{Q})$. We call $\beta$ a~\textit{height difference bound}. 
\begin{cor}\label{C:}
  Let $Q\in A(\Q)_{\tors}$. Then $H(Q)< e^\beta$.
\end{cor}
To compute an explicit bound $\beta$, the standard approach 
is to decompose the difference between the naive height and the canonical
height into local components, see for instance~\cite[Theorem~4]{FlynnSmart}.
As we
shall see, $\beta$ will help us decide whether a $p$-adic torsion point is $\Q$-rational or
not.

\section{An algorithm for finding torsion subgroups of abelian varieties}\label{algorithm}
Let $A/\Q$ denote an abelian variety with Kummer variety $K/\Q$ and a fixed
map $\kappa$ as in~\eqref{kappa}.
In this section we discuss an algorithm which computes the group $\AQt$ as an abstract
abelian group, provided Assumption~\ref{assump} below is satisfied. Our algorithm is based
on an algorithm for Jacobians of genus~2 curves due
to Stoll~\cite[Section~11]{StollG2}.

\begin{ass}\label{assump}
We have algorithms for the following:
    \begin{enumerate}
      \item\label{ass:kappa} the map $\kappa \colon A \rightarrow K\subset \mathbb{P}^{2^g-1}$ and
      equations for its image;
    \item\label{ass:lift} deciding whether a given point $R \in K(\mathbb{Q})$ lifts to $A(\Q)$ under
      $\kappa$;
    \item\label{ass:2B} the maps $[[2]]$ and $B$;
    \item\label{ass:beta} a height difference bound $\beta$;
    \item\label{ass:arithred} arithmetic in the group $\tilde{A}(\mathbb{F}_p)$ for primes of good
        reduction $p$ and enumeration of its elements.  
    \end{enumerate}
\end{ass}
The algorithm in~\cite[Section~11]{StollG2} crucially relies on the fact that Assumption~\ref{assump} is satisfied for Jacobians of curves of genus~2, see~\S\ref{genus2}.
We will show in Section~\ref{genus3} that it is also satisfied for Jacobians of
hyperelliptic curves of genus~3.

\begin{rk}\label{R:arithred}
  We can replace~\eqref{ass:arithred} by the assumption that we also have~\eqref{ass:kappa},
  \eqref{ass:lift}
  and~\eqref{ass:2B} for the reduction $\tilde{K}/\F_p$ if $p$ is a prime of good
  reduction. This is the case, for instance, for Jacobians of hyperelliptic curves of
  genus~$\le 3$ (for $g=3$, we need $p>2$).
  We can then enumerate $\tilde{K}(\F_p)$ and check which elements lift to
  $\tilde{A}(\F_p)$ to compute the latter. Moreover, arithmetic in $\tilde{A}(\F_p)$  can
  be reduced to arithmetic in
  $\tilde{K}(\F_p)$, for which we can use~\eqref{ass:2B}. In practice, we prefer to compute
  (in)   $\tilde{A}(\F_p)$ directly.
\end{rk}

The strategy can be summarised as follows. One
first uses reduction modulo $p$ for a number of good primes $p$ to
obtain an integer $d>0$
such that
$\#\AQt\mid d$. For each prime $q\mid d$, we find the $q$-Sylow subgroup of
$\AQt$; to this end, we first choose a suitable good prime $p\ne q$. For each
$\tilde{Q}\in \tilde{A}(\F_p)$ of $q$-power order $m$, we can compute the
unique lift\footnote{{We hope that no confusion arises from using
the word ``lift''
both for Hensel lifts as well as 
lifts of points from $K$ to $A$.}}
$\tilde{\kappa}(\tilde{Q})$ in  {$\kappa(A(\Q_p)[m])$} to any desired precision $p^N$. Using $\beta$, we 
choose $N$ and construct a lattice $L$ with the following property: If there is a point $R\in
\kappa(A(\Q_p)[m])\cap K(\Q)$ that reduces to
our approximation of $\tilde{\kappa}(\tilde{Q})$ modulo $p^N$, then  the shortest nontrivial vector in
$L$ must
be this point $R$. We can decide whether such a point exists by applying
the LLL algorithm. If it does, then it 
remains to check whether it  lifts to $A(\Q)[m]$. See Algorithm~\ref{lifttorsionpoints} for
odd $q$. This is then used in Algorithm~\ref{qpartoftorsion}, which computes the $q$-part of
$A(\Q)_{\tors}$ for odd $q$. The
case $q=2$ is discussed in~\S\ref{subsec:2pow}.
Finally, Algorithm~\ref{torsionsubgroup}  computes $A(\Q)_{\tors}$, provided
Assumption~\ref{assump} is satisfied.

\begin{rk}\label{preimages}
We stress that we do not assume that we can explicitly compute in $A(\Q)$; nor do we assume
that we can explicitly write down points in $A(\Q)$.
  If the latter is possible and if we can compute the preimages under $\kappa$
  in~\eqref{ass:lift}, 
  then we can also find $\AQt$ as a set, see Remark~\ref{R:kummer-to-jac} below.
\end{rk}


\subsection{Checking whether reduced points lift}\label{algorithmliftcheck}
The most challenging part of the algorithm is to check whether
a reduced torsion point lifts to a rational torsion point or not. More
specifically, given a prime $p$ of good 
reduction and a point $\tilde{Q} \in \tilde{A}(\mathbb{F}_p)$ of order
$m$ coprime to $p$, there 
exists a unique lift $Q \in A(\mathbb{Q}_p)[m]$
such that $Q$ reduces to $\tilde{Q}$. This algorithm decides whether
$Q \in A(\mathbb{Q}) \subset A(\mathbb{Q}_p)$. 
Below, we will apply the LLL-algorithm with standard parameter
$\delta=\frac{3}{4}$ (see~\cite{LLLreduction}).

\begin{algo}\label{lifttorsionpoints} \textbf{Lifting Torsion Points}\\
  Input: An abelian variety $A/\Q$ {such that Assumption~\ref{assump} is
  satisfied} and 
    a point $\tilde{Q} \in A(\mathbb{F}_p)$ of order $m > 2$ coprime to $p$.
\\  Output: TRUE if  there is a
  point $Q \in
  A(\mathbb{Q})_{\tors} \subset A(\mathbb{Q}_p)_{\tors}$ that
  reduces to $\tilde{Q}$, else FALSE.

\begin{enumerate}
  \item\label{liftbeta} {Compute a 
    height difference bound $\beta$ for $A$}.
  \item\label{lift1} Choose $M = 1 + am$ such that $p\nmid a$.
  \item\label{liftR0} Let $\tilde{R}_0$ be $\tilde{\kappa}(\tilde{Q})$, considered on
		  an affine patch in
                  $\mathbb{A}^{2^g}(\mathbb{Z}/p\mathbb{Z})$ and
		  normalised such that the first nonzero coordinate 
		  is equal to $1$. Set $r \colonequals 1$, $n \colonequals 0$.
                
                \item\label{liftN} Let $N>1$ such that $p^N > 2^{(2^g + g)}e^{2\beta}$. While $r < N$, repeat the following steps:
                  \begin{enumerate}
          \item\label{liftr} Set $r\colonequals \min\{2r, N\}$.
          \item\label{liftRnt} Let $\tilde{R}_n'$ be any lift of
            $\tilde{R}_n$ to $\mathbb{A}^{2^g}(\mathbb{Z}/p^r\mathbb{Z})$.
          \item\label{liftRn+1} Set $\tilde{R}_{n+1} \colonequals \frac{1}{M -
            1}(M\tilde{R}_n' - [[M]](\tilde{R}_n'))$, where $M\tilde{R}_n'$
				is obtained by multiplying the coordinates of $\tilde{R}_n'$ by $M$.
                              \item\label{liftrn+1} Set $n \colonequals  n + 1$.
	\end{enumerate}
      \item\label{liftRn} Now, consider $\tilde{R}_n \equalscolon (\tilde{r}_1 : \ldots : \tilde{r}_{2^g})$
		  in $K(\mathbb{Z}/p^N\mathbb{Z})$.
		  Let $(r_1, \ldots, r_{2^g}) \in \mathbb{Z}^{2^g}$
		   reduce to $(\tilde{r}_1, \ldots, \tilde{r}_{2^g})$ modulo $p^N$ such
                   that $0 \leq r_i < p^N$ for all $i$. Let
		  $L$ be the lattice generated by $(r_1, \ldots, r_{2^g})$
                  and by $(p^Ne_1, \ldots, p^Ne_{2^g})$, where
                  $(e_1,\ldots,e_{2^g})$ is the standard basis
		  of $\mathbb{Z}^{2^g}$. Let $w$ be the first basis vector
                  of an $\textrm{LLL}$-reduced basis of $L$ and let
		$R=\BP w$ be the corresponding point in $\mathbb{P}^{2^g-1}(\mathbb{Q})$.
              \item\label{liftht} If $R \notin K(\mathbb{Q})$ or $H(R) > e^\beta$, return FALSE.
              \item\label{lift0} If $[[m]](R)\ne \kappa(0)$, return FALSE.
              \item\label{liftlift} If $\kappa^{-1}(R) \subset A(\mathbb{Q})$, return TRUE.
		  Else return FALSE.
\end{enumerate}
\end{algo}
We prove the correctness of the 
algorithm in~\S\ref{algorithmconclusions}.
For Jacobians of curves of genus~2, this is sketched in
Stoll~\cite[Section~11]{StollG2}.
\begin{thm}\label{prooflifttorsionpoints}
  Algorithm \ref{lifttorsionpoints} terminates and returns TRUE if and only if there is a
  point $Q \in
  A(\mathbb{Q})_{\tors} \subset A(\mathbb{Q}_p)_{\tors}$ that
  reduces to $\tilde{Q}$.
\end{thm}
We first need some preliminary results.

\subsubsection{The lifting procedure}\label{liftingprocedureproof}
We start by showing that Step~\eqref{liftN}
of Algorithm \ref{lifttorsionpoints} lifts to the $m$-torsion point that we 
want to approximate.
We say that a point on $K$ is~\emph{$m$-torsion} if the map $[[m]]$
sends it to $\kappa(0)\in K$. Equivalently, a point on $K$ is $m$-torsion
if and only if there is a point in $A[m]$ that maps to it under $\kappa$.

\begin{prop}\label{step3proof}
  After Step~\eqref{liftN} of Algorithm \ref{lifttorsionpoints}, $\tilde{R}_n$ is the 
unique $m$-torsion point in $K(\mathbb{Z}/p^N\mathbb{Z})$ that reduces to $\kappa(\tilde{Q})$.
\end{prop}

In order to prove Proposition~\ref{step3proof}, we first show that
Step~\eqref{liftRn+1} approximates $Q$ by an $m$-torsion lift to the required $p$-adic precision
$p^N$. 
By \cite[III, \textsection 8, Corollary 2]{Bourbaki} and 
\cite{pAdicAbVar},
the group $A(\mathbb{Q}_p)$ is a $p$-adic abelian Lie group
whose topology is the local product topology: a neighborhood
of a point $Q \in A(\mathbb{Q}_p)$ is a neighborhood $U$ of $Q$ contained in an affine
space, and for any $d\ge 1$, the $p$-adic topology on
$\mathbb{A}^d(\mathbb{Q}_p) = \mathbb{Q}_p^d$  is induced
by the maximum norm $ \|\cdot\|_p$.

\begin{lemma}\label{approx}
Let $Q \in A(\mathbb{Q}_p)$ be a torsion point of order $m$, not divisible by $p$.
Let $n\ge 1$, let $\phi \colon A \rightarrow \mathbb{A}^n$ be a rational
  map defined over $\mathbb{Q}_p$ that is differentiable as a map
  $A(\Q_p)\to \mathbb{A}^n(\Q_p)$ and  a $p$-adic immersion near $Q$,
and let $a \in \mathbb{Z}$. If $U \subset A(\mathbb{Q}_p)$ is a
neighborhood of $Q$, then
for any $Q' \in U$, we have
\begin{equation}\label{approxeq}
  \phi([1 + am]Q') - \phi(Q) = (1 + am)(\phi(Q') - \phi(Q)) +
  {O}(\|\phi(Q') - \phi(Q)\|_p^2)\,.
\end{equation}
\end{lemma}

\begin{proof}
  For the proof, we set $M \colonequals 1 + am$, so that $[M](Q) = Q$. 
Near $Q$, the map $\phi$ is an immersion, so there is a well-defined map
  $[[M]]$ that makes the diagram 
\begin{equation}\label{immersionM}
\centering
\begin{tikzcd}
A(\mathbb{Q}_p) \arrow[r, "\lbrack M \rbrack"] \arrow[d, "\phi"]& A(\mathbb{Q}_p) \arrow[d, "\phi"] \\
\phi(A(\mathbb{Q}_p)) \arrow[r, "\lbrack \lbrack M \rbrack \rbrack"]& \phi(A(\mathbb{Q}_p))
\end{tikzcd}
\end{equation}
commute on a neighborhood of $Q$.
  Since $\phi$ is a rational map to $\A^n$, we have that $\phi(A(\mathbb{Q}_p))$ consists of
the $\mathbb{Q}_p$-rational points on an affine variety over $\mathbb{Q}_p$. 
  Hence the differential of  $[[M]]\colon \phi(A)(\Q_p)\to \phi(A)(\Q_p)$ at $\phi(Q)$ is the best
  linear approximation of $[[M]]$ around $\phi(Q)$. 
  In other words, 
it consists of the linear terms of the Taylor expansion of $[[M]]$ 
around $\phi(Q)$.
By \cite[Chapter III, \textsection 2.2]{Bourbaki} the differential of the multiplication-by-$M$-map $[M]$ is
  scalar multiplication on the tangent space, and a computation
  using~\eqref{immersionM} shows that the same holds
  for the differential of $[[M]]$.

Now let $Q'$ be close to
$Q$, so that $\phi(Q')$ is close to
$\phi(Q)$. By the above, we find
  $$[[1 + am]](\phi(Q')) - [[1 + am]](\phi(Q)) = (1 + am) (\phi(Q') - \phi(Q)) + O( \|\phi(Q') - \phi(Q)\|_p^2).$$
Using \eqref{immersionM}, we have
  $[[1 + am]](\phi(Q)) = \phi([1 + am](Q)) = \phi(Q)$. Therefore (\ref{approxeq})
follows.
\end{proof}

We now apply Lemma \ref{approx} to a map $\phi \colon A\to \A^{2^g}$ that
factors through $\kappa \colon A \rightarrow K$.

\begin{proof}[Proof of Proposition \ref{step3proof}]

Since $\kappa$ is differentiable outside $A[2]$, 
composing  with a map that projects onto an affine patch results
  in a differentiable map that is a local immersion outside $A[2]$. Let $\phi$ denote 
the map $\kappa$ composed with the projection onto a suitable affine patch.
Then $\phi$ satisfies the conditions of Lemma 
\ref{approx} and we obtain
  $$[[M]](\tilde{R}_n') - \tilde{R}_{n+1} = M(\tilde{R}_n' - \tilde{R}_{n+1}) + O(\|\tilde{R}_{n+1} - \tilde{R}_n'\|_p^2).$$
  By construction, we have  $\|\tilde{R}_{n+1} - \tilde{R}_n'\|_p^2 =
  p^{-r}$ in Step~\eqref{liftN}
of Algorithm \ref{lifttorsionpoints}, and therefore
\[
  \tilde{R}_{n+1} = \frac{1}{M-1}(M\tilde{R}_n' - [[M]](\tilde{R}_n')) + O(p^r)
\]
is the $m$-torsion point in $K(\mathbb{Z}/p^r\mathbb{Z})$ that reduces to $\kappa(\tilde{Q})$.
\end{proof}
\begin{rk}
Intuitively, one can view $\phi$ as a map that gives local affine coordinates
of $Q$ with the property that we can find a best linear approximation of the
multiplication-by-$(1 + am)$-map.
For the approximation in Step~\eqref{liftRn+1}, one may use a different projection onto $\mathbb{A}^{2^g}$ 
  in every iteration of Step~\eqref{liftN}. This may be necessary if, for example, the first
coordinate of $R$ is divisible by $p^r$, but is not divisible by $p^{2r}$
  for some $r \geq 1$.
\end{rk}


\begin{rk}\label{elimbqf1}
In \cite[\textsection 11]{StollG2}, it is assumed that $p$ divides $M$.
Here, we generalise this by allowing $M \not\equiv 1 \mod p$.
One way to use  this additional flexibility in practice is to choose $M$ to be a power of
  $2$, because doubling on $K$ is often faster than applying
  the map $B$. See~\S\ref{avoidbqf}.
  \end{rk}

\subsubsection{Determining a suitable $p$-adic precision}\label{algorithmicheightbound}
We now show  that we can find a $p$-adic precision such that the 
corresponding rational approximation $\tilde{R}_n\in K(\Z/p^N\Z)$ either leads to a rational
lift $R = \kappa(Q)$ such that $Q \in A(\mathbb{Q})$, or no such rational 
lift exists.

\begin{prop}\label{LLLjustif}
  Let $N \in \mathbb{Z}$ be such that $p^N > 2^{(g + 2^g)}e^{2\beta}$. Let $\tilde{R}_n$, $(r_1, \ldots, r_{2^g})$,
  $L$, $w$ and $R$ be as computed in Step~\eqref{liftRn} of Algorithm \ref{lifttorsionpoints}.
Then we have:
\begin{enumerate}
  \item[(a)] If $H(R) \leq e^\beta$, then $R$ is the unique point in
    $\mathbb{P}^{2^g-1}(\mathbb{Q})$ that satisfies $H(R)\le e^\beta$ and
	 reduces to $\tilde{R}_n$.
       \item[(b)]  If $H(R) > e^\beta$, then no point on $\mathbb{P}^{2^g-1}(\mathbb{Q})$
         of height $\leq e^\beta$ reduces
	to $\tilde{R}_n$.
\end{enumerate}
\end{prop}

To lift points, we use the following result, whose proof is immediate.
\begin{lemma}\label{latticeclarification}
Let $n,d \in \mathbb{Z}_{\geq 1}$ and let $\tilde{R} \in \mathbb{P}^d(\mathbb{Z}/p^n\mathbb{Z})$.
  Let 
\[
v \colonequals (r_0, \ldots, r_{d}) \in \mathbb{Z}^{d+1}\setminus\{0\}
\]
be primitive, i.e.  $\gcd(r_0, \ldots, r_{d}) = 1$, such that
  $R\colonequals \BP v \colonequals (r_0 : \ldots : r_{d}) \in
  \mathbb{P}^d(\mathbb{Q})$ lifts $\tilde{R}$. 
Then the lattice $L$ generated by $\{v\} \cup \{e_ip^n : 0 \leq i \leq d\}$
contains all vectors $w$ such that the corresponding point $\BP w\in \mathbb{P}^d(\mathbb{Q})$
reduces modulo $p^n$ to $\tilde{R}$. Moreover, let
\[
u \colonequals a_0v + p^na_1e_1 + \cdots p^na_{d+1}e_{d+1} \in L\,,
\]
 where $a_0,\ldots,a_{d+1} \in \mathbb{Z}$. If  $p\nmid a_0$,
then  $\BP u\in \mathbb{P}^d(\mathbb{Q})$ reduces modulo
$p^n$ to $\tilde{R} \in \mathbb{P}^d(\mathbb{Z}/p^n\mathbb{Z})$.
\end{lemma}
Moreover, we need the following uniqueness result.
\begin{lemma}\label{L:inj}
  Let $B\ge 1$ be a real number and let $d\in\Z$ be positive.
  Let $u,u'\in \Z^{d+1}\setminus\{0\}$ such that
  \begin{enumerate}[(a)]
    \item\label{inja} $\|u\|_\infty, \|u'\|_\infty \le B$
    \item\label{injb} there is an integer $D>2B^2$ such that all $2\times 2$ minors of the matrix $\begin{pmatrix}
    u\\u'\end{pmatrix}\in \Z^{2\times (d+1)}$ are divisible by $D$.
  \end{enumerate}
  Then the points $\BP u$ and $\BP u'$ in $\BP^{d}(\Q)$ represented by $u$
  and $u'$, respectively, are equal.
\end{lemma}
\begin{proof}
  By~\ref{inja}, the $2\times 2$ minors have absolute value bounded by
  $2B^2$. Hence they all vanish by~\ref{injb}. Therefore $u=\lambda u'$
  for some nonzero $\lambda\in \Q$.
\end{proof}
We apply Lemma~\ref{L:inj} to the lattice $L$ from Step~\eqref{liftRn} of
Algorithm~\ref{lifttorsionpoints}. We thank an anonymous referee for
suggesting the structure of the following proof.

\begin{proof}[{Proof of Proposition~\ref{LLLjustif}}]
By Lemma~\ref{latticeclarification}, the lattice $L$ 
  in Step~\eqref{liftRn} contains all integer 
representatives of the points in $\mathbb{P}^{2^g-1}(\mathbb{Q})$ that reduce
to  $\tilde{R}_n \in \mathbb{P}^{2^g-1}(\mathbb{Z}/p^N\mathbb{Z})$ as
  obtained after Step~\eqref{liftN} of Algorithm \ref{lifttorsionpoints}.
   Moreover, any vector that corresponds to a lift of $\tilde{R}_n$
is of the form $u=a_0v + p^Na_1e_1 + \cdots p^Na_{2^g}e_{2^g}$ with $p\nmid
  a_0$.
  For such a vector $u$, we have 
  \begin{equation}\label{htnorm}
    H(\BP u)\le \|u\|_\infty,\quad\text{with equality if and only if }u
    \text{ is primitive.}
  \end{equation}

  Let $w\in \Z^{2^g}$ be the first vector of an LLL-reduced basis of $L$ as
  in  Step~\eqref{liftRn} of Algorithm \ref{lifttorsionpoints}. 
  We distinguish cases as follows. First suppose that
  $\|w\|_\infty>2^{(2^g-1)/2}\sqrt{2^g}e^\beta$.
  In this case, we claim that there is no
  nonzero vector $u\in L$ such that $\|u\|_\infty\le e^\beta$. 
  Indeed, our choice $\delta=3/4$ of parameter in the LLL-algorithm implies
  that the first basis vector of an LLL-reduced basis has euclidean length 
  at most $2^{(2^g-1)/2}$ times the euclidean length of 
  the shortest nonzero vector (see~\cite{LLLreduction}). 
  In particular, there is no point in $\BP^{2^g-1}(\Q)$ of height $\le
  e^\beta$ that reduces to $\tilde{R}_n$ by
  Lemma~\ref{latticeclarification} and by~\eqref{htnorm}.

  Now suppose that $\|w\|_\infty\le 2^{(2^g-1)/2}\sqrt{2^g}e^\beta$.
  By construction, all pairs of nonzero vectors in $L$ satisfy 
  condition~\ref{injb} of Lemma~\ref{L:inj} with $D=p^N>2e^{2\beta}$. 
  Since $2 \cdot (2^{(2^g-1)/2}\sqrt{2^g}e^\beta)^2 =
  2^{2^g}(2^g)e^{2\beta} < p^N$, 
  we can apply Lemma~\ref{L:inj} with $d=2^g-1$ and
  $B=2^{(2^g-1)/2}\sqrt{2^g}e^\beta$.
  This implies that for any
  vector $w'\in L$ satisfying $\|w'\|_\infty\le B$, we have $\BP w' = \BP w
  = R$. Hence, by
  Lemma~\ref{latticeclarification} and by~\eqref{htnorm},
  if there is a point in $\BP^{2^g-1}(\Q)$ that reduces to
  $\tilde{R}_n$ and has height $\le e^\beta$, then it must be $R$. 
\end{proof}

\subsubsection{The conclusions of the lift-checking algorithm}\label{algorithmconclusions}


\begin{proof}[Proof of Theorem~\ref{prooflifttorsionpoints}]
It is clear that the algorithm terminates. 
To prove Theorem~\ref{prooflifttorsionpoints}, it suffices to prove the
correctness of Steps~\eqref{liftht}--\eqref{liftlift}
of Algorithm \ref{lifttorsionpoints},  
  which we do now.
  Since a torsion point $P\in A(\Q)$ satisfies $H(P) = H(\kappa(P))\le
  e^\beta$, Proposition~\ref{step3proof} and Proposition~\ref{LLLjustif}
  imply that if there is a point
  $Q\in A(\Q)[m]$ that reduces to $\tilde{Q}$, then the point $R\in K(\Q)$ from
  Step~\eqref{liftN} of Algorithm~\ref{lifttorsionpoints} satisfies
  $R=\kappa(Q)$. 
  Clearly we then have $[[m]](R)=\kappa(0)$ and
  $\kappa^{-1}(R) = \{\pm Q\}\subset A(\Q)$, so that the algorithm returns
  TRUE.

  Conversely, suppose that the algorithm returns TRUE. Then, using
  Propositions \ref{step3proof} and \ref{LLLjustif} again, $R$
  is the unique $m$-torsion point in $K(\Q)$ that reduces to
  $\tilde{\kappa}(\tilde{Q})$. Thus the points in $\kappa^{-1}(R)$ are
  $\Q$-rational points of order $m$, so one of these two points is a point
  in $A(\Q)[m]\subset A(\Q)_{\mathrm{tors}}$ that reduces to $\tilde{Q}$.
\end{proof}

\begin{rk}\label{earlytermination}
  In practice, we can often terminate the algorithm long before the
  required precision in Step~\eqref{liftN} is reached, as follows: 
  Let $\tilde{R}_r$ be as in Step~\eqref{liftN}, for some $r<N$. From
  $\tilde{R}_r$, determine $R$ using Step~\eqref{liftRn} and check if the conditions of
  Step~\eqref{liftht}--\eqref{liftlift} are satisfied.
If they are, then we have found a point $Q \in A(\mathbb{Q})[m]$ that
reduces to $\tilde{Q}$. However, if no such point is found, then it is not guaranteed
that no other candidate exists.
\end{rk}

\subsection{Computing the rational torsion subgroup}\label{computingtorsionsubgroup}
Now that we can conclusively decide for good primes $p$  whether a point in 
$\tilde{A}(\mathbb{F}_p)$ 
lifts to $A(\mathbb{Q})_{\tors}$ or not, we can find the rational
torsion subgroup of $A(\Q)$. Since we do not assume that we can represent or compute with
general points in $A(\Q)$, we compute $\AQt$ as an abstract abelian group by finding
its invariant factors.
This is again a generalisation of the idea proposed in \cite[\textsection
11]{StollG2} for Jacobians of curves of genus~2.
For a prime $q$ and a finite abelian group $G$, we let the \textit{$q$-part of $G$} be the
$q$-Sylow subgroup of $G$, as an abstract abelian group.  Then the
reduction map $\rho_{p}\colon A(\Q_p)\to A(\mathbb{F}_p)$ is injective 
on $q$-parts of $A(\mathbb{Q})_{\tors}$ for any prime number $q \neq p$.

\begin{algo}\label{qpartoftorsion} \textbf{Computing the $q$-part of the torsion subgroup}\\
  Input: an abelian variety $A/\Q$ for which Assumption~\ref{assump} is
  satisfied and a prime $q > 2$.
  \\
  Output: The $q$-part
  of $A(\mathbb{Q})_{\tors}$ as an abstract abelian group.

\begin{enumerate}
  \item\label{qp1} Let $G_0$ be the $q$-part of $\tilde{A}(\mathbb{F}_p)$, where $p$ is a
          good prime not equal to $q$.
		  Set $T_0 \colonequals  \{0\} \subset G_0$, $S_0
                  \colonequals  G_0 \setminus \{0\}$, $S_0' \colonequals 
                  \{0\}$. (Throughout, $G_i$ is a quotient of $G_0$, $T_i$
                  is a subgroup of $G_0$, and
    $S_i$ and $S'_i$ are subsets of $G_i$.)
                \item\label{qp2} Set $n \colonequals  0$, repeat the following steps until $S_n = \emptyset$.
                  \begin{enumerate}
          \item\label{qp21} Let $g \in S_n\subset G_n$ and choose a
            representative $\tilde{g}\in G_0$ of $g$.
                \item\label{qp22} Using Algorithm~\ref{lifttorsionpoints}, compute the
                  smallest {$\ell>0$} such that $q^\ell \tilde{g}$ lifts to $A(\mathbb{Q})$. 
                \item\label{qp23} Set
			\begin{align*}
				T_{n+1} &\colonequals  \langle T_n, q^\ell
                                \cdot \tilde{g}
                                \rangle\,,\\
				G_{n+1} &\colonequals  G_n/\langle q^\ell  \cdot g
                                \rangle\,,\\
                                S_{n+1}' &\colonequals\; \textrm{image
                                of}\;  S_n' \cup \langle g
                                \rangle\;\text{in}\; G_{n+1}\,,\\
				S_{n+1} &\colonequals  G_{n+1} \setminus S_{n+1}'\,.
			\end{align*}
                      \item\label{qp24} Set $n \colonequals n + 1$.
	\end{enumerate}
      \item\label{qp3} Return $T_n$ as an abstract abelian group.
\end{enumerate}
\end{algo}


  It is preferable to take a primitive element in Step~\eqref{qp21}, but
  this is not required. In Step~\eqref{qp1}, we typically pick a prime 
$p$ such that the $q$-part of $\tilde{A}(\mathbb{F}_p)$ is small. If it is
  trivial, then there is nothing to do. In practice, we have already computed
  $\tilde{A}(\mathbb{F}_p)$ for all good primes below some bound, see
  Algorithm~\ref{torsionsubgroup} below.

\subsubsection{Two-power torsion}\label{subsec:2pow}
Algorithm \ref{lifttorsionpoints} excludes the case $m = 2$ 
because the lifting procedure does not work on points of order $2$, since $\kappa(A[2])$ consists of singular points. 
  It is still possible to 
compute $A(\Q)[2]$, for instance by finding the solutions $R\in K(\Q)$
  of the projective system of equations
  $[[2]](R) = \kappa(0)$ and checking which of these lift . Hence we
  can skip this case in Algorithm~\ref{qpartoftorsion}.
We can,
alternatively, determine $A(\Q)[2^\infty]$ iteratively as follows: For
$s\ge2$, we find $A(\Q)[2^s]$ from $A(\Q)[2^{s-1}]$ for $s
\ge 2$ by solving the system $[[2]](R)
= S$ for each $S\in
\kappa(A(\Q)[2^{s-1}])$ and checking which solutions lift.
We implemented this strategy for Jacobians of hyperelliptic
curves of genus~3, but we found that this is quite expensive. Fortunately, in this case  there is a simpler method, discussed in~\S\ref{g3twotorsion}.

\subsubsection{The algorithm}

\begin{algo}\label{torsionsubgroup} \textbf{Computing the Torsion Subgroup}\\
  Input: an abelian variety $A/\Q$  satisfying Assumption~\ref{assump}. 
  \\Output: the invariant factors of
  $A(\mathbb{Q})_{\tors}$.\par
\begin{enumerate}
  \item\label{t1} Compute a height difference  bound $\beta$.
  \item\label{t2} Compute a multiplicative upper bound $t$ for the size of the torsion 
	subgroup by computing the
        structure of 
        $\tilde{A}(\mathbb{F}_p)$ for a reasonable number of good odd primes $p$.
      \item\label{t3} For each prime factor $q$ of $t$, compute the $q$-part of
        $A(\mathbb{Q})_{\tors}$
          using Algorithm \ref{qpartoftorsion} and~\S\ref{subsec:2pow}.
        \item\label{t4} Deduce the invariant factors of 
          $A(\mathbb{Q})_{\tors}$ from the invariant factors of its
          $q$-parts.

\end{enumerate}
\end{algo}

\begin{rk}\label{R:kummer-to-jac}
  If we can describe points in $A(\Q)$ explicitly  and if we have an algorithm to compute
  $\kappa^{-1}(R)$ for given $R\in K(\Q)$, then we can also return $\kappa^{-1}(R)$ in
  Step~\eqref{liftlift} of Algorithm~\ref{lifttorsionpoints}. 
  In this case, we can also compute (generators for) the $q$-part 
  in Algorithm~\ref{qpartoftorsion}, rather than only its structure as an abstract abelian
  group. Hence we can amend Algorithm~\ref{torsionsubgroup} to find generators for
  $A(\Q)_{\tors}$.
\end{rk}

\subsection{Avoiding the use of sum-and-difference-laws}\label{avoidbqf}
In practice, one of the most expensive tasks in
Algorithm~\ref{torsionsubgroup} is the computation of $[[n]](R)$ for points $R\in K$ and
potentially large $n\in \Z$.
Namely, in Algorithm~\ref{lifttorsionpoints},
we apply $[[M]]$ in Step~\eqref{liftRn+1}
and we apply $[[m]]$ in Step~\eqref{lift0}.
Recall from Algorithm~\ref{kummerarithmetic} that the
multiplication-by-$n$-map $[[n]]\colon K\to K$ requires formulas for the 
doubling map $[[2]]\colon K\to K$ and for the map
 $B\colon \Sym^2(K)\to \Sym^2(K)$ such that 
 $$B(\{\kappa(Q_1),\kappa(Q_2)\}) = \{\kappa(Q_1 +
Q_2), \kappa(Q_1 - Q_2)\}$$
for $Q_1,Q_2\in A$.
For Jacobians of hyperelliptic curves of genus $\le3$, the formulas for the
map $B$ are much  more
complicated than those for the map $[[2]]$.
Hence, we prefer to apply the map $[[n]]$ only 
for small $n$ of the form $n = \pm 2^s$ since 
then the doubling formulas suffice.
In addition, we might be in a situation where the doubling map $[[2]]$ is available
explicitly, but the map $B$ is not. Then it turns out that it is often still  possible to 
compute $A(\Q)_{\tors}$, as we now explain.

Recall that by construction, $m$ is a power of
a prime $q \geq 2$. In most cases, $m$ will be small.
We require $M$ to satisfy $M \equiv 1 \bmod m$
and $M \not\equiv 1 \bmod p$, so we can use 
 $M = 1 - m$ if we want to keep $M$ small. 
If $m$ is odd, it is clear that we can instead find a suitable $M$ of the form $M=\pm 2^s$,
and Step~\eqref{liftRn+1} of Algorithm~\ref{lifttorsionpoints} can be performed using only the map $[[2]]$.
This does not work when $m$ is even. However, recall that we can use the strategy
discussed in~\S\ref{subsec:2pow} to compute the 2-part of $A(\Q)_{\tors}$ without
Step~\eqref{liftRn+1} of Algorithm~\ref{lifttorsionpoints}.

Besides Step~\eqref{liftRn+1}, arithmetic on $K$ is also used  
in Step~\eqref{lift0} of Algorithm \ref{lifttorsionpoints}. Here, we check whether 
a point $R \in K(\mathbb{Q})$ satisfies $[[m]](R) = \kappa(0)$. If
arithmetic in $A(\mathbb{Q})$ is implemented, for instance when $A$ is the
Jacobian of a hyperelliptic curve of even genus or odd degree, then
we can avoid Step~\eqref{lift0} by first computing $\kappa^{-1}(R)\cap
A(\mathbb{\Q)}$. If this set is
non-empty, say containing 
a point $Q$, then  we can check directly whether $mQ= 0 \in A(\mathbb{Q})$. 
If no algorithm for arithmetic in $A(\mathbb{Q})$ is available, then we can only avoid the use of the
map $B$ in Step~\eqref{lift0} for specific values of $m$.
For instance, suppose that all prime powers $m$ dividing $t$ in Algorithm~\ref{torsionsubgroup} are at
most~60. Then we can avoid the use of $B$ if and only if all these $m$ satisfy
$m \in \{2^u : u \in \mathbb{Z}_{\geq 1}\} \cup \{3,9,5,7,17,31\}\,, $
and if $t$ is not divisible by both~7 and~9. See~\cite[\S4.7]{Rei20} for details.

\subsection{Computing torsion subgroups for Jacobians of genus 2 curves}\label{genus2}
Suppose that $A=J$ is the Jacobian of a curve $X/\Q$ of genus~2 and let $K$ denote its
Kummer surface. We may assume that $X$ is given by an equation $y^2=f(x)$,
where $f\in \Q[x]$ is squarefree and has degree 5 or 6.
If $\deg(f)=5$, then we can represent points on $J$ using the (affine) Mumford
representation. More generally, points
in $J(\Q)$ correspond
bijectively to 
triples $(A,B,C)$ of binary forms over $\Q$ of homogeneous degrees~2, 3 and 4, respectively, such that
the degree~6 homogenisation $F$ of $f$ satisfies $F = B^2-AC$
(see~\cite{BS10}).
One can use this representation to compute in the group $J(\Q)$ via a generalisation of
Cantor's algorithm~\cite{Can87}. In fact, Cantor's algorithm has
been extended to any curve of genus 2 over any field.

Assumption~\ref{assump} is satisfied for $J$:
\begin{itemize}
    \item A morphism $\kappa\colon J\to \BP^3$ such that $\kappa(J)$ is a model for $K$
      was given by Flynn~\cite{Fly93}, see also~\cite[Chapter~3]{Prolegomena}. In this case the Kummer surface
    is a quartic hypersurface.
  \item A point in $K(\Q)$ lifts to $J(\Q)$ if and only if the expressions in
    Equations~(5.1, 5.2) of~\cite{Sto02} are squares in $\Q$.
  \item The map $[[2]]$ is given by quartic polynomials and $B\colon
    \Sym^2(K)\to \Sym^2(K)$ is given by
    biquadratic forms; explicit formulas can be found in~\cite[Section~3]{Prolegomena}.
  \item There is an explicit theory of heights which allows us to compute a height
    difference bound $\beta$; see~\cite{Fly95, FlynnSmart, StollG2, MullerStollG2}.
\end{itemize}
Hence Algorithm~\ref{torsionsubgroup} can be used to compute $\#J(\Q)_{\tors}$.  
In fact, one can compute (in) $\tilde{J}(\F_p)$ for primes $p$ of good reduction using the
(generalised) Mumford representation, which is faster than the approach in
Remark~\ref{R:arithred}.
Moreover, we can compute 
$J(\mathbb{Q})[2]$ easily using the prime
factorisation of $f$ in $\mathbb{Q}[x]$, see \cite[Lemma 4.3, Lemma 5.6]{Stoll2desc}.

Using the generalised Mumford representation, we
can actually compute generators of $J(\Q)_{\tors}$. As mentioned above, this is essentially already
discussed in~\cite[\S 11]{StollG2} and an implementation is available in~{\tt Magma}.

\section{Computing torsion subgroups of Jacobians of genus 3 hyperelliptic curves}\label{genus3}
Section \ref{algorithm} gives a complete algorithm to compute the torsion
subgroup for an abelian variety that satisfies Assumption~\ref{assump}. In this section, we
show that Assumption~\ref{assump} is satisfied 
when {$A=J$} is the Jacobian of a hyperelliptic curve of genus~3. Hence we obtain an
algorithm to compute $J(\Q)_{\tors}$, which we have implemented in {\tt Magma} and which
is available at~{\url{https://github.com/bernoreitsma/g3hyptorsion}}. This answers
a question raised by Andrew Sutherland at the 2017 Banff Workshop ``Arithmetic Aspects of
Explicit Moduli Problems''.

Throughout this section, we fix a field $k$ such that $\mathrm{char}(k)\ne
2$ and a hyperelliptic curve $X/k$ of genus~3 given by an equation
\begin{equation*}
  X\colon y^2 = f(x)\,,
\end{equation*}
where $f\in k[x]$ is squarefree of degree~7 or~8. 
Let $\iota\colon X\to X$ be the hyperelliptic involution and let $J/k$ be the Jacobian of
$X$. We will represent (most) points on $J$ using the following notion:
\begin{defn}\label{genpos}
A divisor $D $ on $X$ is \emph{in general position} if it is effective and
if there
  is no point $P \in X$ such that $D \geq (P) + \iota (P)$. 
\end{defn}

In the literature, the explicit theory of hyperelliptic curves
is usually first developed for the case where the polynomial $f$
has odd degree. More generally, if $X(k)$ contains a Weierstrass point, then we may apply
a transformation to get an odd degree equation over $k$.
In this case, every point on the Jacobian can be represented uniquely by a divisor of the form
$D-d(\infty)$, where $d\le 3$ and $D$ is in 
in general position.
This
leads to the unique Mumford representation $(a,b)$ of a point $Q\in J(k)$, where $a\in
k[x]$ is monic of degree~$d$ and vanishes precisely in the $x$-coordinates of the points
in $\supp(D)$, and $b\in k[x]$ determines the $y$-coordinates. 
The Mumford representation can be used to perform arithmetic in $J(k)$ 
using Cantor's algorithm~\cite{Can87}.
Based on this, an explicit theory of the Kummer variety was found for the degree~7 case
in~\cite{Stu00, MulExplKum3}.\par

For our application, we do not assume that $X$ contains a $k$-rational Weierstrass point (or, in
fact, any $k$-rational point). 
Instead, we rely on an explicit theory of the Kummer variety in the general case developed and implemented  by Stoll (see~\cite{StollG3, StollMAGMA}).
We summarise his results here and describe a few modest additions of ours.

\subsection{Representing points on the Jacobian}\label{g3divrep}
In order to find an explicit map $\kappa \colon J \rightarrow \mathbb{P}^7$ such
that $\kappa(J)$ is a model of $K$, we need
an explicit description of points on $J$ without the assumption $\deg(f)=7$.  We will now show that we can represent points $Q$ on $J$ using
divisors of degree $4$, but we cannot expect uniqueness anymore.\par

We follow the discussion in~\cite{StollG3}. The idea is to use the
canonical isomorphism between $\Pic^0(X)$ and $\Pic^4(X)$ given by adding the
canonical class. Let the divisor $D_\infty$ on $X$ be equal to $2(\infty)$ if $\deg(f)=7$ and to
$(\infty_1)+(\infty_2)$ otherwise, where $\infty_1$ and $\infty_2$ are the two points at
infinity on $X$. Then $2D_\infty$ is a canonical divisor of $X$.

\begin{prop}\label{d_q}
For every $Q \in J\setminus\{0\}$, exactly one of the following holds
  \begin{enumerate}[(a)]
    \item\label{dqgen} $Q = [D_Q - 2D_\infty]$ for a divisor
    $D_Q$ of degree~4 in general position;
  \item\label{dqspec}  $Q = [D_Q - D_\infty]$ for a divisor
    $D_Q$ of degree~2 in general position.
\end{enumerate}
  In case~\ref{dqspec}, the divisor $D_Q$ is uniquely determined by $Q$.
\end{prop}
\begin{proof}
  The proof is sketched in \cite[\textsection 2]{StollG3}. 
  Suppose that $Q\in J\setminus\{0\}$ is represented by $E_Q\in \Div^0(X)$.
  By Riemann--Roch, the Riemann--Roch space $L(E_Q+2D_\infty)$ has
  dimension at least~2; let $\varphi$ be a nontrivial element. Then
  $D'_Q\colonequals E_Q+2D_\infty+\div(\varphi)$ is effective of degree~4
  and we have $Q = [D'_Q-2D_\infty]$. If $D'_Q$ is in general position,
  then we set $D_Q\colonequals D'_Q$; otherwise $D'_Q = D_Q+(P) +
  (\iota(P))$ for some $P \in X$ and $D_Q$ is effective of degree~2 and in
  general position, since 
 $Q\neq 0$.

  For $D_Q$ of degree~2 in general position, the
  Riemann--Roch space $L(D_Q+D_\infty)$ has
  dimension~2, generated by $1$ and $x$. Hence all elements of the
  corresponding linear system, containing all divisors linearly
  equivalent to $D_Q+D_\infty$, are of the form $D_Q+(P) + (\iota(P))$
  for some $P\in X$. This   shows that the two
  cases \ref{dqgen} and~\ref{dqspec} are mutually exclusive and also
  proves the uniqueness of $D_Q$ in case~\ref{dqspec}.
%
\end{proof}
From now on, we say that $Q$ is of \emph{degree $4$} in case \ref{dqgen} and 
of \emph{degree $2$} in case~\ref{dqspec}. The point $0 \in J$
is defined to have degree $0$.\par

We first consider the case where $Q \in J$ has degree $4$. 
Then the divisor $D_Q$ in Proposition~\ref{d_q}\ref{dqgen} is not
unique by~\cite[Lemma~2.1]{StollG3}. 
As in~\S\ref{genus2}, $D_Q$
yields  a generalised Mumford representation as follows.
Let $F$ be the degree~8 homogenisation of $f$. There is a model
$  y^2=F(x,z)$
of $X$ in the weighted projective plane over $k$ with weight~1 associated to $x$ and $z$
and weight~$4=g+1$ associated to $y$.
By~\cite[page~4]{StollG3}, divisors $D\in \Div^4(X)$ in general position correspond bijectively to triples of
binary forms $A, B, C\in k[x,z]$ of degree $4$ such that
\begin{equation}\label{mumfordvariety}
B^2 - F = AC\,.
\end{equation}
The image of a point $P=(x_0:y_0:z_0)$ in the support of $D$ under the hyperelliptic
covering $\pi \colon X \rightarrow \mathbb{P}^1$ corresponds to a root of $A$ with the
correct multiplicity, and we have $y_0=B(\pi(P))$. Note that this Mumford
representation of $D_Q$ is unique up to adding multiples of $A$ to $B$.

 \begin{rk}\label{lcfsq}
   If $X(k)$ is non-empty, then we can find an equation $y^2=f(x)$ for $X$ such that
 we  either have $\deg(f)=7$ or we have
   $\deg(f)=8$ and the  leading coefficient of $f$ is a square. We have already discussed
   the former case. In the latter case, we can
arbitrarily fix one of the two points $\infty_1,\infty_2\in X(k)$ at infinity, say
$\infty_1$. If $Q\in J(k)$ has degree~4, then requiring that $\infty_1\in \supp(D_Q)$ fixes $D_Q$ uniquely. 
By the above, we can represent $Q$ using a triple $(A,B,C)$ representing
   $D_Q$. Moreover, we can use
   this representation for arithmetic in $J(k)$ using a generalisation of Cantor's
   Algorithm. This is implemented in {\tt Magma}. In practice, it is better
   to use 
   Sutherland's balanced divisor approach~\cite{sutherland}, which is
   more efficient. It also requires the existence of a $k$-rational point.
 \end{rk}

If the leading coefficient of $f$ is not a square in $k$, then it is not clear how to
represent degree~4 points consistently (and hence uniquely). In this case, {\tt Magma} does
not represent such points and arithmetic in $J(k)$ has not been implemented.

\subsection{The Kummer variety}\label{subsec:Kumg3}
In~\cite[Lemma 2.1]{StollG3}, Stoll shows that there is a subgroup $\Gamma$
of $\mathrm{SO}(Q)$, where $Q$ is the ternary quadratic form $y^2-xz$,  with the following property: Two triples $(A,B,C)$ and
$(A',B',C')$ represent 
divisors in general position of degree~4 with the same image on $J$
if and only if they are equivalent under the 
action of $\Gamma$. Moreover, they represent inverse points 
if and only if they are equivalent under the action of
 $-\Gamma$. 
 Stoll then uses this observation to construct
the Kummer variety $K$ of $J$ explicitly as follows. 
There is a canonical theta divisor $\Theta$ on $J$ such that the support of
$\Theta$ consists of $0$ and
the points on $J$ of degree~2 in the sense of Proposition~\ref{d_q}.
A basis for the Riemann-Roch space $\mathcal{L}(2\Theta)$ defines a rational map $\kappa\colon J\to \BP^7$ such that $\kappa(J)$ is a model for the Kummer variety $K$ of $J$.
By the above, the complement of the image of $\Theta$ in $\Pic^4$ under the canonical
isomorphism can be described by the affine variety $V$ defined 
by~\eqref{mumfordvariety}, 
quotiented out by the action of $\Gamma$. 
Stoll finds a basis $\xi_1,\ldots,\xi_8$ of $\mathcal{L}(2\Theta)$ from $\pm \Gamma$-invariants in $k[V]$. 
Let 
$\kappa \colon J \rightarrow \mathbb{P}^7$ be the map defined by
$\xi_1,\ldots,\xi_8$.
Then $\kappa$ is invariant under multiplication by $-1$ on $J$. 
Hence its image $K \colonequals \kappa(J)$ describes a birational model of the Kummer variety by~\cite[Theorem 2.5]{StollG3}.

According to~\cite[Proposition 3.1]{MulExplKum3}, $K$ can be 
			defined by quartic relations.
To find such relations, Stoll notes that $\xi_1,\ldots,\xi_7$ are of degree~2 in the
coefficients of $A,B$ and $C$, whereas $\xi_8$ is
quadratic in $\xi_1,\ldots,\xi_7$, leading to a quadratic relation, and hence 36
quartic ones, satisfied by the
$\xi_i$. By~\cite[Theorem 2.5]{StollG3}, one needs an additional~34 quartic relations;
such relations are constructed before~\cite[Lemma~2.2]{StollG3}.

To describe the map $\kappa$ on points $Q\in J(k)$ of degree~2 (which lie on $\Theta$),
Stoll
approximates the divisor $D_Q+D_\infty$, where $D_Q$ is as in
Proposition~\ref{d_q}\ref{dqspec} (see the discussion following~\cite[Theorem
2.5]{StollG3}). 
Write $D_Q = (P_1)+(P_2)$, where $P_i = (x_i : y_i : z_i)\in X$, and
$$A(x,z) = (z_1x-x_1z)(z_2x - x_2z)\equalscolon a_0x^2 + a_1xz +
a_2z^2\in k[x,z]\,.$$
Then we have
\begin{equation*}
\kappa(Q) = (0 : a_0^2 : a_0a_1 : a_0a_2 : a_1^2 -
  a_0a_2 : a_1a_2 : a_2^2 : \xi_8)\,.
\end{equation*}
If $z_1=z_2=1$ and  $x_1\ne x_2$, then $a_0=1$ and  
$\xi_8= \frac{2y_1y_2 - G(x_1,x_2)}{(x_1 - x_2)^2}$, where 
\begin{equation*}
G(x_1, x_2) = 2\sum_{j=0}^4 f_{2j}(x_1x_2)^j + (x_1 + x_2)\sum_{j=0}^3
  f_{2j + 1}(x_1x_2)^j
\end{equation*}
and $F(x,z) = f_0z^8+f_1xz^7+\ldots+f_8x^8$. 
In this case, $\xi_8$ satisfies
\begin{equation}\label{expansionequation}
((x_1 - x_2)^2\xi_8 + G(x_1, x_2))^2 - 4f(x_1)f(x_2) = 0\,.
\end{equation}
We now give $\kappa(Q)$ explicitly for the remaining special cases. More details can be
found in~\cite[\S5.4]{Rei20}.
If $z_1=z_2=1$ and $x_1=x_2$, then we can find $\xi_8$  by writing~\eqref{expansionequation}  as 
\begin{equation}\label{sigmaexpansionsum}
s_2\xi_8^2 + s_1\xi_8 + s_0 = 0.
  \end{equation}
Then $s_2=0$ and $s_1 = -2G(x_1, x_1) = -4f(x_1)$. If $s_1\ne 0$, then
$Q\ne 0$,
and it follows that \begin{equation}\label{explicitdeg2mapmult2}
  \kappa(Q) = \left(0 : 1 : -2x_1 : x_1^2: 3x_1^2 : -2x_1^3 : x_1^4 :
  \frac{-s_0}{s_1}\right).
\end{equation}

If $z_1=1$ and $P_2 = (1 : w : 0)$ for some $w \in \bar{k}$ such that
$w^2 = f_8$, then we can use an approximation to find 
\begin{equation}\label{explicitdeg1map}
\kappa(Q) = (0 : 0 : 0 : 0 : 1 : -x_1 : x_1^2 : 2y_1w - 2f_8x_1^4 - f_7x_1^3).
\end{equation}

If 
$P_1 = P_2 = (1 : w : 0)$, then we can use~\eqref{sigmaexpansionsum} to find
\begin{equation}\label{2infty}
\kappa(Q) = (0: 0 : 0 : 0 : 0 : 0 : 4f_8 : 4f_6f_8 - f_7^2).
\end{equation}

Finally, we have
\begin{equation*}
\kappa(0) = (0:0:0:0:0:0:0:1).
\end{equation*}

\begin{rk}\label{R:}
  If $s_2=0$, then we can also express $\xi_8$ in terms of the
  coefficients of the polynomials $A,B,C$: 
  \begin{equation*}
    \xi_8 = -a_0^3c_6 - a_0^2a_2c_4 + 2a_0^2b_2b_4 - 2a_0a_1b_1b_4 -
    a_0a_2^2c_2 + 2a_0a_2b_1b_3 + 2a_1^2b_0b_4 - 2a_1a_2b_0b_3 - a_2^3c_0 +
    2a_2^2b_0b_2
  \end{equation*}
  where $B(x,z) = b_0z^4+b_1xz^3+b_2x^2z^2+b_3x^3z+b_4x^4$ and
  $C(x,z) = c_0z^6+\ldots+c_6x^6$.  
\end{rk}
\subsubsection{Traces of the group law}\label{subsec:}
Recall that Assumption~\ref{assump} requires, in particular, algorithms for
\begin{itemize}
  \item the map $[[2]]\colon K \to K$ such that $\kappa(2Q) = [[2]](\kappa(Q))$ for all
    $Q\in J$;
  \item  the  map $B\colon \Sym^2(K)\to \Sym^2(K)$ such that for all $Q_1,Q_2\in J$ we
    have $$B(\{\kappa(Q_1),\kappa(Q_2)\}) = \{\kappa(Q_1 + Q_2), \kappa(Q_1
    - Q_2)\}\,.$$
\end{itemize}
Similar to the genus~2 case~\cite[Section~3]{Prolegomena}, there are homogeneous quartic
polynomials $$\delta_1,\ldots,\delta_8\in \Z[f_0,\ldots, f_8][x_1,\ldots,x_8]$$ such that 
$[[2]](R) = (\delta_1(R):\ldots:\delta_8(R))$ for all $R \in K$, normalised to map
$(0,\ldots,0,1)$ to itself. The polynomials can be constructed using representation
theory;
see~\cite[Theorem 7.3]{StollG3}. 
The map $B$ 
is constructed using representation theory in~\cite[Lemma~8.1]{StollG3}.

\subsection{Checking whether rational points lift to rational points}\label{g3KtoJ}
This section gives a procedure that decides whether the preimage under $\kappa$ of a point
\[
R = (\xi_1 : \ldots : \xi_8) \in K(k)
\]
is in $J(k)$ or not. 
Let $Q \in J$ such that $\kappa(Q) = R$. Then $Q$ is of degree
$4$ if and only if $\xi_1 \neq 0$. Also, we have $Q = 0$ if and only if $R = (0:0:0:0:0:0:0:1)$.

The case where $Q \in J$ has degree $4$ is treated in \cite[\textsection 4]{StollG3}.
Briefly, the idea is that when $h$ is a nonzero odd function on $J/k$, then $h^2$ induces a
function $j$ on $K/k$, and $R\in K(k)$ can only have rational preimages if
$j(R)$ is a square in $k$. Conversely, if $j(R)$ is a nonzero square in
$k$, then $R$ has rational preimages.
Stoll constructs suitable functions $j$ as 
$3 \times 3$-minors of a $4\times 4$ matrix $M = M(\xi_1,\ldots,\xi_7)$
(see~\cite[(2.7)]{StollG3}). The preimage of $R$ consists of
rational points if and only if
all values $j(R)$ are squares in $k$. 

Suppose that $Q\in J$ has degree $2$. In this case~\cite[\textsection 4]{StollG3}
suggests to simply consider the map $\kappa$ explicitly. 
 The uniqueness of the divisor $D_Q= (P_1)+(P_2)$ such that $Q = [D_Q-D_\infty]$ implies that 
$Q \in J(k)$ if and only if $D_Q$ is defined over $k$.
By~\S\ref{subsec:Kumg3}, we have $\xi_1=0$. We now distinguish cases, using
the explicit expression for $R$ given in~\S\ref{subsec:Kumg3}. 

First suppose that $\xi_2 =0$. Since $Q\ne O$, we are either in
case~\eqref{explicitdeg1map} or in case~\eqref{2infty}.
If $\xi_5 = 0$, then it is the latter. Note that $\xi_7\ne 0$, since
otherwise we would
have $R = (0:0:0:0:0:0:0:1)$. 
Therefore  $P_1=P_2 =\infty_{1/2}$, and the preimage
of $R$ is rational if
and only if $X$ has rational points at infinity.

If $\xi_2 = 0$, but $\xi_5 \neq 0$, then $R$ is as 
in Equation~\eqref{explicitdeg1map}. Hence without loss of generality $P_1=
(x_1,y_1)$ and $P_2 = \infty_{1/2}$. We have $\kappa^{-1}(R)\subset J(k)$ if and only if
$\infty_{1/2}$ 
and $(x_1,y_1)$ are rational. The latter holds if and only
if $f(-\xi_6) = y_1^2$ is a square in $k$.

It remains to discuss the case $\xi_2 \neq 0$, i.e. $\deg A(x,1)=2$.
Then $ D_Q = (P_1) + (P_2)$, where
$P_i=(x_i,y_i)$ are affine and
\begin{equation}\label{Rxi0}
R=\left(0 : 1 : -(x_1 + x_2) : x_1x_2 : x_1^2 + x_1x_2 + x_2^2 : -(x_1 + x_2)x_1x_2 :
(x_1x_2)^2 : \frac{2y_1y_2 - G(x_1, x_2)}{(x_1 - x_2)^2}\right).
\end{equation}
\begin{lemma}\label{L:deg2rat}
The preimage $\kappa^{-1}(R)$ consists 
of rational points if and only if $y_1+y_2\in k$ and if one of the following conditions
is satisfied:
\begin{enumerate}
  \item $x_1=x_2$,
  \item $x_1\ne x_2$ and $\frac{y_1-y_2}{x_1-x_2}\in k$.
\end{enumerate}
\end{lemma}
\begin{proof}
  The divisor $D_Q$ is $k$-rational if and only if $P_1$ and $P_2$
  are $k$-rational or if $P_1,P_2\in X(k')$ for a quadratic extension $k'/k$
  and $\sigma(P_1)=P_2$, where $\sigma$ is the nontrivial element of
  $\Gal(k'/k)$. Hence a necessary condition for rationality of $D_Q$ is  
  that the polynomial $(y-y_1)(y-y_2)$ is defined over $k$. It follows
  from~\eqref{Rxi0} that $y_1y_2\in k$.
  Suppose from now on that $y_1+y_2\in k$. 

  If $x_1=x_2$, 
then $x_1=-\frac12\xi_3\in k$ and $y_1=\frac12(y_1+y_2)\in k$, since
  $Q\ne 0$ by assumption; hence $D_Q$ is $k$-rational. 
  Now assume that $x_1\ne x_2$. Then 
  there is a Mumford representation $(A,B,C)$ of $Q$ such that the
  polynomial $b(x) = B(x,1)$ is linear and satisfies
  $b(x_1)=y_1$ and $b(x_2)=y_2$, and $D_Q$ is
  $k$-rational if and only if $b\in k[x]$. 
  Write $b(x) = b_0+b_1x$, then $b_1 = \frac{y_1-y_2}{x_1-x_2}$ and $b_0 =
  \frac{y_2x_1-y_1x_2}{x_1-x_2}$. Since
  $$y_1+y_2= 2b_0 +b_1(x_1+x_2)\,,$$ and since both $y_1+y_2$ and
  $x_1+x_2=-\xi_3$
  are in $k$, we conclude that $b\in k[x]$ if and only if $b_1\in k$.
\end{proof}
From~\eqref{Rxi0}, we can compute $y_1y_2$ and 
\begin{equation*}
y_1^2 + y_2^2 	= f(x_1) + f(x_2) 
                                = \sum_{j = 0}^8 f_j(x_1^j + x_2^j)\,,
\end{equation*}
hence also $(y_1\pm y_2)^2$. Since $(x_1-x_2)^2$ is also computed easily
from~\eqref{Rxi0}, we can use Lemma~\ref{L:deg2rat} in practice to
check whether $R$ lifts to rational points.

\begin{rk}\label{R:}
  Stoll shows in~\cite[\S4]{StollG3} how to compute a lift of $R$ when
  $X(k)$ is nonempty and the lifts of $R$ have degree~4. Using the above, we can 
  compute the unique Mumford representation of the points lifting $R$ in the degree~2 case.
\end{rk}

\subsection{Using arithmetic on reduced Jacobians}\label{g3localglobal}
Recall that Step~\eqref{t3} of Algorithm~\ref{torsionsubgroup} requires the
structure of $\tilde{J}(\mathbb{F}_p)$ for
primes of good reduction $p$, where $\tilde{J}$ is the reduction of
$J$ modulo $p$. Moreover, in Step~\eqref{qp22} of Algorithm
\ref{qpartoftorsion}, we need to enumerate all elements 
 of the $q$-parts of $\tilde{J}(\mathbb{F}_p)$, where $q$ is prime and $p\ne q$ is a prime
 of good reduction, and we need to compute scalar multiples.
\par
In Remark~\ref{R:arithred} we discussed how to compute $\tilde{J}(\mathbb{F}_p)$ 
for a prime $p$ of good reduction using arithmetic of the Kummer variety $\tilde{K}$ and checking
whether points in $\tilde{K}(\F_p)$ lift to $\tilde{J}(\F_p)$.
In practice, it turns out to be more efficient to compute $\tilde{J}(\F_p)$ using
arithmetic in $\tilde{J}(\F_p)$, if that is implemented.

Recall from~\S\ref{g3divrep} that there are algorithms (and implementations) for arithmetic
in $\tilde{J}(\F_p)$ if we know a point in $\tilde{X}(\F_p)$, but no algorithm is
known if we do not.
If $\tilde{X}(\mathbb{F}_p)$ is nonempty, then we 
 fix a point $\tilde{P} \in \tilde{X}(\mathbb{F}_p)$ and use a change of coordinates $\phi \colon \tilde{X} \rightarrow \tilde{X}'$
such that $\phi(\tilde{P})$ is a point at infinity. Let $\tilde{J}'$ be the Jacobian of
$\tilde{X}'$. Then we can compute in 
$\tilde{J}'(\mathbb{F}_p)$, for instance in {\tt Magma}. Moreover, we can enumerate
$\tilde{J}'(\mathbb{F}_p)$ and find its structure as an abelian group.
We adjust Steps~\eqref{t3} and~\eqref{t4} of Algorithm \ref{torsionsubgroup}
in the following way. Here, 
we denote the Kummer variety of $\tilde{J}'$ by $\tilde{K}'$.
\begin{itemize}
  \item In Step~\eqref{t3} and~\eqref{t4} of Algorithm \ref{torsionsubgroup}, find suitable primes $p$ 
		  with the extra condition that $\tilde{X}(\mathbb{F}_p)$ is not empty. 
	\item In Algorithm \ref{qpartoftorsion},
          let $G_0$ be the $q$-part of $\tilde{J}'(\mathbb{F}_p)$. In Step~\eqref{qp2},
                  find the smallest $m$ such that $\tilde{\kappa}(\phi_\ast^{-1}(q^m \cdot
                  g))$ lifts to $J(\mathbb{Q})$, where $\phi\colon \tilde{X}\to \tilde{X}'$ is as above.
\end{itemize}
This modification still
allows us to choose from infinitely many primes $p$, since the
Hasse-Weil bound implies
$\#\tilde{X}(\mathbb{F}_p) \geq 1$ for all good $p \geq 41$.

\begin{rk}\label{R:phiK}
In practice, we replace
  $\tilde{\kappa} \circ \phi_\ast^{-1}$
  by $\phi_{\tilde{K}}^{-1} \circ \tilde{\kappa}'$, where 
  $\phi_{\tilde{K}} \colon \tilde{K} \rightarrow \tilde{K}'$ is the isomorphism induced by
  $\phi$. Explicit formulas for $\phi_{\tilde{K}}$ are given in~\cite[Appendix~B]{Rei20}.
\end{rk}

\subsection{Computing the rational two-torsion points.}\label{g3twotorsion}

It is possible to compute $J(\mathbb{Q})[2]$ via the approach sketched 
in~\S\ref{subsec:2pow}, but this takes quite long in practice. We now
discuss a more efficient method, suggested to us by Michael Stoll.

First suppose that $\deg(f) = 7$.
Let $g_1,\ldots, g_r\in k[x]$ be the {monic
  irreducible} factors of $f$.  Then, by~\cite[Lemma 4.3]{Stoll2desc}
$J(\mathbb{Q})[2]$ is generated by the points with Mumford representation
\[
  ( g_1, 0 ), \ldots, (g_{r-1}, 0)\,.
\]
Now suppose that $\deg(f) = 8$. 
Let $\Omega\subset \overline{\mathbb{Q}}$ be the set of zeros of $f(x)$. We
call an unordered partition $\{\Omega_1, \Omega_2\}$ of $\Omega$ \textit{even} if
$\#\Omega_1$ (and hence $\#\Omega_2$) is even.
An even
partition $\{\Omega_1, \Omega_2\}$ of $\Omega$ gives rise to a two-torsion
point $P_{\Omega_1}$ ($= P_{\Omega_2}$) represented by
\begin{equation}\label{ttbijection}
  \sum_{\omega\in \Omega_1} (\omega, 0) - \frac{\#\Omega_1}{2}D_\infty
\sim \sum_{\omega\in \Omega_2} (\omega, 0) - \frac{\#\Omega_2}{2}D_\infty\,,
\end{equation}
and every point in $J(\overline{\mathbb{Q}})[2]$ arises in
this way from a unique unordered partition. See~\cite[\S5]{StollG3}
and~\cite{PoonenSchaefer}. 
More precisely,~\eqref{ttbijection} induces a bijection between $J[2]$ and the Galois-module
of unordered even partitions of roots of $f$
(see~\cite[\S6]{PoonenSchaefer}).
Hence $J(\Q)[2]$ is in bijection with the set of all unordered even partitions $\{\Omega_1,\Omega_2\}$ that are fixed by the absolute Galois group
$G_{\Q}$ of
$\Q$. For instance, if $f$ is monic and $\Omega_1$ (equivalently, $\Omega_2$) is fixed by
$G_{\Q}$, then $P_{\Omega_1} \in J(\Q)[2]$, and $P_{\Omega_1}$ has
Mumford representation $(A(x,z), 0, C(x, z))$, where $A(x,z) = \prod_{\omega\in
\Omega_1} (x-\omega z)$ and $C(x,z) = \prod_{\omega\in
\Omega_2} (x-\omega z)$.

However, in general not every point in $J(\Q)[2]$ arises in this way. It is also possible that $\{\Omega_1,\Omega_2\}$ is fixed by
$G_\Q$, but $\Omega_1$ and $\Omega_2$ are not fixed
individually. Then both $\Omega_1$ and $\Omega_2$ have
size~4 and we have $\Omega_2 = \Omega_1^\sigma$, where
$\sigma$ is the non-trivial element of $\Gal(k/\Q)$ for a quadratic number
field $k$. This corresponds to a factorisation 
$f = \mathrm{lc}(f)\cdot h\cdot h^\sigma$, where $h =\prod_{\omega\in \Omega_1}(x-\omega)\in k[x]-\Q[x]$ and $h^\sigma=
\prod_{\omega\in \Omega_2}(x-\omega)$ are coprime and $\mathrm{lc}(f)$ is
the leading coefficient of $f$. In this case, $P_{\Omega_1}$ has no
Mumford representation of the form $(A,0,C)$ defined over $\Q$ (the
degree-4 homogenisations of $h$ and $h^\sigma$ give such a representative
over $k$). The factorisation $f = \mathrm{lc}(f)\cdot h\cdot
h^\sigma$ implies that $k$ is a subfield of the \'etale algebra
$\Q[x]/(f)$.

We may use this to compute $J(\Q)[2]$ as follows.
Let $2t_\Q$ denote the number of monic even degree divisors of $f$ in
$\Q[x]$. 
For a quadratic extension $k/\Q$ with Galois group $\Gal(k/\Q) = \{1,\sigma\}$, we define 
$$
t_k \colonequals \frac{1}{2}\#\left\{h \in k[x]\,:\, h\; \text{is
monic},\; f =
\mathrm{lc}(f)\cdot h\cdot h^\sigma,\;\gcd(h, h^\sigma)=1 \right\}\,.
$$
By the discussion above, we obtain the following formula.
\begin{lemma}\label{L:deg8tors2}
  We have  $\#J(\Q)[2] = t_\Q + \sum_{k} t_k$, where $k$ runs through the
  quadratic subfields of $\Q[x]/(f)$.
\end{lemma}

\begin{ex}\label{2torsextors3}
  Recall from Example~\ref{extors3} that for the Jacobian $J$ of
\[
  X \colon y^2 = x^8 + 2x^7 + 3x^6 + 4x^5 + 9x^4 + 8x^3 + 7x^2 + 2x + 1\equalscolon f(x)
\]
the group $J(\mathbb{Q})_{\tors}$
is isomorphic to a subgroup of $\mathbb{Z}/6\mathbb{Z}$.
The polynomial $f(x)$
  is irreducible over $\Q$, but it has the factorisation
  $$
      f = (x^4 + (1-2i)x^3 + (-1-2i)x^2 + (-1-2i)x - 1)
  \cdot(x^4 + (1+2i)x^3 + (-1+2i)x^2 + (-1+2i)x - 1)
  $$
  over $\Q(i)$, where $i^2=-1$.
  This shows that $\#J(\Q)[2]=2$.
\end{ex}
\begin{rk}\label{R:}
  Lemma~\ref{L:deg8tors2} holds more generally  for hyperelliptic curves
  over $\Q$ of arbitrary
  genus and even degree. However, if $g$ is even, then $\deg(f)$ is not
  divisible by~4, and hence $t_k=0$ for all quadratic fields $k$. In this
  case all rational 2-torsion points come from even degree factors of $f$
  over $\Q$ and we recover~\cite[Lemma~5.6]{Stoll2desc}.
\end{rk}

\subsection{Halving a rational point on $K$.}\label{halvingpoint}
In practice, the formulas found by Stoll in~\cite[Lemma~8.1]{StollG3} for
the  map $B$ as in Section~\ref{S:kum-heights} need a lot more space to store than the
$\delta_{i}$, and they also take longer to evaluate. 
Recall from~\S\ref{avoidbqf} that we can avoid the $B_{ij}$ altogether
in many situations. 
If $J(\Q)[2]$ is nontrivial (and we do not already know that
$J(\Q)[2^\infty]=J(\Q)[2]$), then this requires computing preimages under $[[2]]$, as
discussed in~\S\ref{subsec:2pow}.
In other words, for $\kappa(Q)=(y_1:\ldots:y_8)\in K(\Q)$ we need to solve a projective system
\begin{equation}\label{deltainv}
  \delta_i(x_1,\ldots,x_n) = cy_i\,,\quad c\in \Q^\times\,,\quad 1\le i \le 8.
\end{equation}
We have implemented this approach in {\tt Magma}, using Gr\"obner bases to find all
rational points on the zero-dimensional projective scheme defined by~\eqref{deltainv} and
the defining equations of $K$.
This approach works in practice, but we found that most of the time, computing such preimages is significantly
slower than simply using the map $B$.

An alternative approach for computing preimages under $[[2]]$ is proposed by Stoll in \cite[\textsection 5]{StollG2}
for genus $2$. We also generalised this to genus~3 and implemented this
generalisation. However, this requires working over the splitting field 
of $f$. Even when $f$ splits completely over $\Q$, we still found the
approach via $B$ to be more efficient.
\subsection{Height difference bound}\label{subsec:betag3}
In~\cite{StollG3}, Stoll describes 
a method to compute $\beta>0$ such that the difference between the naive and the canonical
height is bounded by $\beta$. His approach generalises results for
genus~2~\cite{FlynnSmart, StollG2, MullerStollG2}. 
Stoll shows in~\cite[Corollary 10.3]{StollG3} that one can take
\[
  \beta = \frac{1}{3}|2^6\text{disc}(f)|+\frac{1}{3}\gamma_\infty\,,
\]
where $\gamma_\infty$ is an upper bound for the local height contribution $\eps_\infty$
introduced in~\cite[\S10]{StollG3}. One can find a suitable $\gamma_\infty$ using the
archimedean triangle inequality and representation theory of $J[2]$,
see~\cite[Lemma~10.4]{StollG3}. A refined bound can be obtained by iterating this
procedure~\cite[Lemma~10.5]{StollG3}.


\section{Examples and databases}\label{examples}
We have implemented the algorithm of Section \ref{algorithm} for hyperelliptic
curves of genus $3$ using the explicit theory discussed in Section \ref{genus3} in~{\tt
Magma}.
The implementation is based on Stoll's {\tt Magma}-implementation of explicit formulas for the
Kummer variety and heights available from~\cite{StollMAGMA}.
Our code, as well as the results of the computations discussed below, can be found at  \url{https://github.com/bernoreitsma/g3hyptorsion}. 
We used {\tt Magma v2.6} on a 64-core 2.6 GHz AMD Opteron(TM) Processor 6276 with 256GB RAM,
running  {\tt  Ubuntu 18.04}.

This section provides some example computations, illustrating various aspects of the
algorithm.  We also used our implementation to compute all rational torsion subgroups in a  database maintained by Andrew 
Sutherland \cite{sutherlanddatabase}. Finally, we ran our algorithm on a large number of
hyperelliptic curves of genus~3 with small coefficients. Together with a few additional
constructions, these computations prove Theorem~\ref{T:main}.

\subsection{Example computations}
\begin{ex}
In Example \ref{extors3}, we showed that for the Jacobian $J$ of the curve
\[
X \colon y^2 = x^8 + 2x^7 + 3x^6 + 4x^5 + 9x^4 + 8x^3 + 7x^2 + 2x + 1,
\]
  we have $\#J(\Q)_{\tors}=3$.
  To find a generator using Algorithm~\ref{lifttorsionpoints}, we pick $p = 17$
because the $3$-part of $\tilde{J}(\mathbb{F}_{17})$ is isomorphic to $\mathbb{Z}/3\mathbb{Z}$.
  We choose a point $\tilde{Q}\in \tilde{J}(\F_{17})$ of order $3$ and consider
  $\kappa(\tilde{Q})\in \tilde{K}(\F_{17})$. If the lift $Q \in J(\mathbb{Q}_p)[3]$ of
  $\tilde{Q}$ is indeed in $J(\mathbb{Q})$,
then $\kappa(Q) \in K(\mathbb{Q})$. After a few iterations of the Hensel lifting,
we can check whether the coordinates define a point on $K(\mathbb{Q})$.
Indeed, after computing the power series up to $p^4$, we
  arrive at a point  $R\in K(\mathbb{Q})$ such that $[[3]](R) = \kappa(0)$, and 
  we check that $\kappa^{-1}(R) \subset J(\mathbb{Q})_{\tors}$.
We find that $J(\mathbb{Q})[3]$ is generated by the point represented by the divisor
  $(0:-1:1) - (1:1:0)$, where the points are viewed inside the projective
plane with weights $1,4,1$.
\end{ex}

\begin{ex}\label{newrationalpoint}
The following example was suggested by Andrew Sutherland. Let
$X$ be the hyperelliptic curve over $\mathbb{Q}$ defined by 
\[
y^2 = 5x^8 - 14x^7 + 33x^6 - 36x^5 + 30x^4 + 2x^3 - 16x^2 + 20x - 7.
\]

  The curve $X$ has no small rational points,  
so this example illustrates how
  we can compute $J(\mathbb{Q})_{\tors}$ without an implementation of
the group law in $J(\mathbb{Q})$.
Computing the order of $\#J(\F_p)$ for some small primes of good reduction, we obtain that
  $\#J(\Q)_{\tors}\mid 13$,
 but no rational point 13-torsion point on $J$ is found
  easily.

For our algorithm, we pick the prime
of good reduction $p = 3$, resulting in the curve 
\[
  \tilde{X}\colon \tilde{y}^2 = 2\tilde{x}^8 + \tilde{x}^7 + 2\tilde{x}^3 + 2\tilde{x}^2 + 2\tilde{x} + 2
\]
over $\F_3$,
which is isomorphic over $\F_3$ to
\[
  \tilde{X}'\colon \tilde{y}^2 = \tilde{x}^8 + \tilde{x}^7 + \tilde{x}^6 + 2\tilde{x}^3 +
  \tilde{x}^2 + 2\,.
\]
  Since $\tilde{X}'$ has rational points at infinity, arithmetic in
  $\tilde{J}'(\F_3)$ is implemented, see the discussion in~\S\ref{g3localglobal}. As in
  Remark~\ref{R:phiK} we use the induced change of coordinates on the Kummer varieties of
  $\tilde{J}$ and $\tilde{J}'$ to check whether a candidate point $\kappa(\tilde{Q})\in
  \tilde{J}(\F_3)[13]$ lifts to $J(\Q)_{\tors}$.
We indeed find the point 
\[
  R = (0 : 1 : -1 : 1 : 0 : -1 : 1 : 20) \in \kappa(J[13])\cap K(\Q)
\] and we can show that 
  $\kappa^{-1}(R) \subset J(\mathbb{Q})$.
  Therefore we have $J(\Q)_{\tors}\cong \Z/13\Z$.

  Since the first coordinate of $R$ is~0, the preimages $Q\in J(\Q)$ of  $R$ are of degree~2
  and hence can be described uniquely using a divisor $D_Q-D_\infty$. 
  A short calculation using the explicit formulas in~\S\ref{g3KtoJ} shows that one of the
  points $Q$ has 
  $$
D_Q = (1 + \zeta_3 , 1 + \zeta_3 ) + (1 + \zeta_3^2 , 1 + \zeta_3^2 ),
  $$ where $\zeta_3$ is a primitive third root of unity.

Alternatively, one can search for points of
  bounded height on $J(\Q)$  reducing to $\tilde{\kappa}(\tilde{Q})$ using a lattice-based
  approach as in~\cite[\S11]{StollG3}. This 
  also finds a rational point of order~13. 
\end{ex}

\begin{ex}\label{kronbergex}
According to~\cite[Example 3.9]{kronbergPhD},
the curve $X$ defined by 
{\tiny
\[
y^2 = \frac{46656}{3125}x^7 + \frac{407097961}{39062500}x^6 + \frac{281238453}{3906250}x^5 - 
\frac{22959453}{312500}x^4 - \frac{2767361}{15625}x^3 + \frac{381951}{2500}x^2 +
  \frac{3093}{6250}x + \frac{1}{2500}
\]}

has a torsion point of order $41$.
  It is easy to see that~41 is an upper bound for $\#J(\Q)_{\tors}$.
  We run our algorithm on the curve with equation $y^2=f(x)$, where
  {\tiny
  \[
f=583200000x^7 + 40709761x^6 + 2812384530x^5 - 2869931625x^4 -
  6918402500x^3 + 5967984375x^2 + 19331250x + 15625\,.
\]}

  The height difference bound $\beta$ computed using Stoll's code
  satisfies 
  $\beta \approx 97$, hence we need $N \log(p) \geq 11\log(2) + 194$ in Step~\eqref{liftN}
  of Algorithm \ref{lifttorsionpoints}. We pick $p = 7$; this yields
the required $p$-adic precision $O(p^N)$ where $N = 128$, which is reached
  in just $7$ steps in Step~\eqref{liftN}. It turns out that we need not go that far; 
$N = 32$ suffices to find a 
  lift $R \in K(\mathbb{Q})\cap \kappa(J[41])$.
  After showing that $R=\kappa(Q)$ for some $Q\in J(\Q)$, we see that 
  $J(\mathbb{Q})_{\tors}\cong \Z/41\Z$, confirming~\cite[Example 3.9]{kronbergPhD}. 
 An explicit generator is
  represented by 
  $  (0 , 125 ) - (\infty)$.
We checked that the Jacobian is in fact 
geometrically simple using the results of 
  \cite[\textsection 3]{HoweZhu}; this was also done by Nicholls
  using~\cite[Proposition~2.4.2]{JacobianDescent}.
\end{ex}

\subsection{Sutherland's database}\label{subsec:suth}
Using the techniques of~\cite{BSSVY16},
Andrew Sutherland has assembled a file with 67879 genus 3 hyperelliptic curves 
of small discriminant at \cite{sutherlanddatabase}.
We used our implementation to compute the rational torsion subgroups of their
Jacobians. For the complete
database containing the results for the 67879 curves, we refer to the file 
\texttt{database.txt} in \url{https://github.com/bernoreitsma/g3hyptorsion}. 
All torsion structures
and the frequency of their appearance can be found in 
Table~\ref{results_suth}. 
Column \textit{inv factors} contains the invariant factors,~\textit{ord} the order of the
group, \textit{count} is the number of times we found this torsion structure and
\emph{gs?}\ indicates whether we found at least one curve whose Jacobian has this torsion
structure and is geometrically simple.
\begin{table}[!h]
\begin{tabular}{|l|l|l|l|}
\hline
  inv factors & ord & count & gs? \\ \hline
  1 & 1& $38370$ &yes\\ \hline
  2 &  2&17093 &yes \\ \hline
  3 & 3& 956 & yes\\ \hline
  2, 2 & 4& {2483} &yes \\ \hline
  4 & 4& 2673 &yes \\ \hline
  5 & 5& 616 &yes \\ \hline
  6 & 6& 1332 &yes \\ \hline
  7 & 7& 701 &yes \\ \hline
  2, 2, 2 & 8& 163 &yes \\ \hline
  2, 4 & 8&493 &yes \\ \hline
  8 & 8&639 &yes \\ \hline
  9 & 9&175 &yes \\ \hline
  10 & 10& 493 &yes \\ \hline
  11 & 11& 34 &yes \\ \hline
  2, 6 & 12& 161 &yes \\ \hline
  12 &  12&403 &yes \\ \hline
  13 &  13&22 &yes \\ \hline
  14 &14 &307 &yes \\ \hline
  15 & 15 &5 &yes \\ \hline
  2, 2, 2, 2 & 16& 3 &yes \\ \hline
  2, 2, 4 & 16& 47 &yes \\ \hline
  2, 8 & 16& 156 &yes \\ \hline
  4, 4 & 16& 3 &no \\ \hline
  16 & 16& 57 &yes \\ \hline
\end{tabular}
\quad
\begin{tabular}{|l|l|l|l|}
\hline
  inv factors & ord & count & gs? \\ \hline
  17 & 17& 5 & yes\\ \hline
  3, 6 &18& 1 &yes \\ \hline
  18 &18&30 &yes \\ \hline
  19 & 19 & 3 &yes \\ \hline
  2, 10 &20 & 88 &yes \\ \hline
  20 &20 & 33 &yes \\ \hline
  21&21 & 2 &yes \\ \hline
  22 & 22 & 14 &yes \\ \hline
  2, 2, 6 & 24 & 19 &yes \\ \hline
  2, 12 & 24 & 98 &yes \\ \hline
  24 & 24 & 21 &yes \\ \hline
  25 & 25 & 4 &yes \\ \hline
  26 & 26 & 9 &yes \\ \hline
  27 & 27 & 3 &yes \\ \hline
  2, 14 & 28 & 33 &yes \\ \hline
  28 & 28 & 17 &yes \\ \hline
  30 & 30 & 6 &yes \\ \hline
  2, 2, 2, 4 & 32 & 1 &yes \\ \hline
  2, 2, 8 & 32 &21 &yes \\ \hline
  2, 16 & 32 & 10 &yes \\ \hline
  32 & 32 & 3 &yes \\ \hline
  2, 18 & 36& 2 &yes \\ \hline
  3, 12 & 36& 1 &no \\ \hline
  6, 6 & 36& 2 &no \\ \hline
\end{tabular}
\quad
\begin{tabular}{|l|l|l|l|}
\hline
  inv factors &ord & count & gs? \\ \hline
  36 & 36 & 3 &yes \\ \hline
  37 & 37 & 1 &yes \\ \hline
  38 & 38 & 2 &yes \\ \hline
  2, 2, 10 & 40& 16 &yes \\ \hline
  2, 20 & 40 & 7 &yes \\ \hline
  40 & 40 & 1 &no \\ \hline
  42 & 42 & 6 & yes \\ \hline
  2, 22 & 44 &1 &  yes\\ \hline
  44 & 44 & 1 &  yes\\ \hline
  46 & 46 & 1 &  yes\\ \hline
  2, 2, 12 & 48 & 2 &no\\ \hline
  2, 24 & 48 & 7 &yes\\ \hline
  4, 12 & 48& 2 &no\\ \hline
  49 & 49 & 2 &yes \\ \hline
  5, 10 & 50&1 &no \\ \hline
  52 & 52 & 2 & yes \\ \hline
  2, 2, 14 & 56 & 1 &yes \\ \hline
  2, 28 & 56 & 4 & yes\\ \hline
  2, 30 & 60 & 1 &no\\ \hline
  60 & 60 & 3 &no\\ \hline
  2, 2, 2, 8 & 64 & 1 & yes \\ \hline
  2, 6, 6 & 72 & 1 &no \\ \hline
  2, 52 & 104 & 1 &no \\ \hline
     &&&\\\hline
\end{tabular}

\caption{Torsion structures found in the database~\cite{sutherlanddatabase}}\label{results_suth}
\end{table}

Here we summarise some of our findings.
\begin{itemize}
	\item $38370$ Jacobians $(\approx 56.5\%)$  have trivial rational torsion subgroup.
	\item $5663$ Jacobians ($\approx 8.3\%$) have a rational torsion point of odd
          order.
	\item $25679$ Jacobians ($\approx 37.8\%$) have a nontrivial cyclic rational 
			torsion subgroup, hence $3830$ ($\approx 5.6\%$) have $2$ or more generators.
	\item Of the non-cyclic torsion subgroups found, $3555$ have $2$
          generators, $370$ have $3$ generators,
			and $5$ torsion subgroups have $4$ generators. The $5$ curves 
			that have four generators all have at least $3$ of these generators of order $2$. 
	\item $11$ Jacobians have a torsion
			subgroup such that there are two invariant factors that are not equal to $2$. 
                      \item 
                      For  $65938$ $(\approx 97.1\%)$ of the Jacobians,
                        the order of the rational torsion subgroup is
                        equal to the upper bound $b$ obtained by reducing
                        modulo all good primes below~1000 as in
                        Example~\ref{extors3}. For the others,
                        we have the following, where \textit{count} denotes
                        the number of occurrences.
Most of the Jacobians for which the quotient $b/\#J(\Q)_{\mathrm{tors}}$ is
    not 1 are geometrically split, for instance, all Jacobians
    for which the quotient is $>7$ or equal to~6, and 182 out of the 192 Jacobians
    with quotient equal to~4.
    The three Jacobians for which the quotient is~7 are geometrically
    irreducible; they have upper bound~7 and $\#J(\Q)_{\mathrm{tors}}=1$.
                        \bigskip

    \begin{tabular}{|l|l|l|l|l|l|l|l|l|l|l|}
  \hline
  $b/\#J(\Q)_{\mathrm{tors}}$& 2& 3& 4&  5& 6&  7&  8&   
  10& 16& 32\\
  \hline
  count &
  1644& 56& 192& 2& 8& 3& 25& 1&  9& 1\\
  \hline
\end{tabular}
\end{itemize}
\medskip

\subsection{Large orders}\label{subsec:large}
\subsubsection{Previous work}\label{subsec:}
In \cite[Table 3.2]{JacobianDescent}, Nicholls lists all known orders of rational
torsion points on 
Jacobians of hyperelliptic curves 
of genus $3$. Most of these were constructed by him in suitable families; in particular, he constructs
geometrically simple Jacobians $J/\Q$ with a point $P\in J(\Q)$ 
of order $N$ for every $N\in \{25,\ldots, 44\}$. Moreover, he constructs such points for 
\[
  N \in \{15, 22, 48, 49, 50, 52, 54, 56, 64, 65, 72, 91\}\,.
\]
In particular, the Jacobians of the curves
  \begin{align}\label{nichols431} 
    y^2 &= -16x^7 + 409/4x^6 - 275x^5 + 399x^4 - 334x^3 + 160x^2 - 40x +
    4\\
    y^2 &= -16x^7 + 393/4x^6 - 237x^5 + 309x^4 - 242x^3 + 116x^2 - 32x+
    4\label{nichols432}
  \end{align}
  have a rational point of order 43. This is the largest known prime order for a rational
  point on the Jacobian of a hyperelliptic curve
of genus $3$ (the previous record holder was the curve in Example~\ref{kronbergex}). 
The largest known point order is~91, but Nicholls does not give the equation of the curve.

\begin{rk}\label{R:}
We focused on geometrically simple Jacobians. In~\cite[\S4.3--\S4.6]{SplitJacobians}, Howe, Leprevost and Poonen construct split Jacobians of
hyperelliptic curves of genus~3 with  large torsion orders. They find the groups with the
following invariant factors:
$$
  [2,30], [10,10], [2,8,8],
    [2,2,2,24], [2,2,2,4,8], [2,2,6,12],
       [4,4,8],
        [2,2,2,4,8], [2,2,2,2,4,8]
        $$
\end{rk}

\subsubsection{Searching for large orders}\label{subsec:search}
Howe~\cite{How15} searched among genus~2 curves of the form
\begin{equation}\label{How}
y^2 +h(x)y=g(x)
\end{equation}
with $\deg(h)=3$ and $\deg(g)=2$ and small coefficients to find large torsion orders. Such curves are promising,
because every curve of genus~2 with a rational non-Weierstrass point has a model of the
form~\eqref{How}.

Similarly, we naively searched among those genus~3 curves that have a model
\begin{equation*}
y^2 +h(x)y=g(x)
\end{equation*}
with $\deg(h)=4$, $\deg(g) =3$ and coefficients bounded in absolute value by~8.
See the file \texttt{searchresults.m} at
\url{https://github.com/bernoreitsma/g3hyptorsion}.

We found the following~3 pairwise non-isomorphic curves having $\#J(\Q)_{\tors}=43$:
\begin{align*}
  y^2&=    x^8 + 4x^6 + 12x^5 - 4x^4 + 24x^3 + 20x^2 - 16x + 16\\
  y^2&=   x^8 - 4x^7 + 10x^5 + 4x^4 - 20x^3 + x^2 + 12x + 4\\
  y^2&=    x^8 - 4x^7 + 18x^5 - 16x^4 - 12x^3 + 9x^2 + 8
  \end{align*}
  The third curve is isomorphic to the curve~\eqref{nichols431} found by
  Nicholls. We did
  not recover the example~\eqref{nichols432} and we found no larger prime order.
All three Jacobians are geometrically simple. 

The largest order $\#J(\Q)_{\tors}$ that we found was~160; this occurred exactly once, for
the following curve, whose Jacobian is geometrically simple: 
$$
y^2 =  9x^8 - 48x^7 + 46x^6 + 96x^5 - 119x^4 - 72x^3 + 64x^2 + 24x\,.
$$
This is the largest torsion order on a geometrically simple Jacobian of dimension~3 found
so far.

The largest (finite) order of an element of $J(\Q)$
was on the Jacobian $J$ of the curve defined by
$$y^2 = 9x^8 - 36x^7 + 36x^6 + 18x^5 - 48x^4 + 24x^3 + x^2 - 4x + 4\,.$$
We have $J(\Q)_{\tors}\cong \Z/144\Z$. Here $J$ is not geometrically simple. 
The largest (finite)  order of a rational point on a geometrically simple
Jacobian  occurred for the curve
\[
  y^2=      x^8 - 2x^7 + 7x^6 - 6x^5 - x^4 + 10x^3 - 6x^2 + 1
\]
whose Jacobian has $J(\Q)_{\tors}\cong \Z/91\Z$, generated by the point
$[2(1,2)-D_{\infty}]$.

Table~\ref{results_search} contains all group structures found in the search which do not already appear for a
geometrically simple Jacobian of a curve in Sutherland's database.
\begin{table}[H]
\begin{tabular}{|l|l|l|}
\hline
  inv factors &ord &  gs? \\ \hline
\hline
      3,3 & 9 &yes \\\hline
      4,4 & 16 &yes \\\hline
      23 & 23 &yes \\\hline
     5, 5 & 25 & yes \\\hline
     29 & 29& yes\\\hline
     31 & 31& yes\\\hline
     2, 4, 4 & 32 &no \\\hline
     4, 8 & 32& yes \\\hline
     35 & 35 & yes \\\hline
     3, 12 & 36 & yes \\\hline
     6, 6 & 36 &no \\\hline
     39 &39 &yes \\\hline
     40 &40 & yes \\\hline
     41 & 41 & yes \\\hline
     43 & 43& yes\\\hline
     2, 2, 12 & 48 &yes \\\hline
     4, 12 & 48& yes\\\hline
     48 & 48& yes\\\hline
    5, 10 & 50 & no \\\hline
     50 & 50& yes\\\hline
     51 & 51 & yes \\\hline
 \end{tabular}
\quad
\begin{tabular}{|l|l|l|l|}\hline
  inv factors &ord &  gs? \\ \hline
\hline
     2, 26 & 52 &yes \\\hline
  3, 18 & 54 &yes \\\hline
     54 & 54 &yes \\\hline
     56 &56 &yes \\\hline
     58 &58 &yes \\\hline
     2, 30 &60 &yes \\\hline
     60 &60 &no \\\hline
     63 & 63 &yes \\\hline
     2, 2, 16 &64 &no \\\hline
          2, 4, 8 &64 &yes \\\hline
     2, 32 &64 &yes \\\hline
     4, 16 & 64&no\\\hline
     64 & 64&yes \\\hline
     65 & 65&yes \\\hline
     70 & 70&yes \\\hline
     2, 2, 18 &72 &no\\\hline
     2, 6, 6 &72 &no \\\hline
     2, 36 &72 &no \\\hline
     6, 12 &72 &yes \\\hline
     72 & 72& yes\\\hline
     2, 2, 2, 10 & 80&yes \\\hline
\end{tabular}
\quad
\begin{tabular}{|l|l|l|l|}\hline
  inv factors &ord &  gs? \\ \hline
\hline
     2, 2, 20 &80 &yes \\\hline
     2, 40 & 80& no\\\hline
     4, 20 & 80&no \\\hline
     80 &80 &no \\\hline
     2, 42 &84 &yes \\\hline
     2, 44 &88 &yes \\\hline
     91 & 91& yes\\\hline
     2, 2, 24 &96 &no \\\hline
     2, 4, 12 &96 &no\\\hline
     2, 48 &96 &no\\\hline
     4, 24 & 96&no \\\hline
     2, 2, 28 &104 &yes \\\hline
     2, 52 &104 &yes \\\hline
     2, 60 & 120 &no \\\hline
     2, 2, 2, 2, 8 & 128 &no \\\hline
     2, 2, 2, 16 & 128&no \\\hline
     2, 4, 16 &128 &no \\\hline
     12, 12 &144 &no \\\hline
     144 &144&no\\\hline
     2, 2, 2, 2, 10 & 160 & yes\\\hline
     &&\\\hline
\end{tabular}
\caption{Torsion structures found in the search}\label{results_search}
\end{table}
\begin{rk}\label{R:new_orders}
  In our computations, we found all point orders in Nicholls' \cite[Table
  3.2]{JacobianDescent}. Moreover, the following
  orders appeared for geometrically simple Jacobians, but were not previously described in
  the literature for
  such Jacobians: 
  \begin{equation*}
23,24, 46, 51, 58, 63, 70
  \end{equation*}
In addition, we found every order up to~22.
  We also found the following new orders for split Jacobians:
  \begin{equation*}
    60,80,144
  \end{equation*}
  The corresponding curves all have automorphism group of order greater
  than~2, so their Jacobians are split over $\Q$. 
\end{rk}
\subsection{Additional examples and proof of Theorem~\ref{T:main}}\label{subsec:proof}
All torsion structures in Theorem~\ref{T:main} occurred in  the
computations discussed
in~\S\ref{subsec:suth} and~\S\ref{subsec:search} (see Table~\ref{results_suth} and
Table~\ref{results_search}), {except for $(\Z/2\Z)^5, (\Z/2\Z)^6, (\Z/2\Z)^4\times \Z/4\Z$
and $(\Z/2\Z)^3\times \Z/6\Z$.
  It is easy to find geometrically simple Jacobians with
    rational torsion subgroup isomorphic to the first two using ~\S\ref{g3twotorsion}. 
    For instance, the curves 
\[
  X_1\colon y^2 = x(x-1)(x-2)(x-3)(x-4)(x^2+x+1)
\] 
and
\[
X_2\colon y^2 = x(x-1)(x-2)(x-3)(x+1)(x+2)(x+3)
\] 
have geometrically simple Jacobian with rational torsion subgroup isomorphic to $(\Z/2\Z)^5$ and
$(\Z/2\Z)^6$, respectively.
In a systematic search, we also found the curves
\[
X_3\colon y^2 = 
 x^7 - 8x^6 - 19x^5 + 235x^4 - 130x^3 - 875x^2 - 500x
\] 
and
\[
X_4\colon y^2 = 
 x^7 - 15x^6 + 87x^5 - 244x^4 + 335x^3 - 191x^2 + 9x + 18
\] 
whose Jacobians $J_3$ and $J_4$ are geometrically simple. We have
$J_3(\Q)_{\tors}\cong (\Z/2\Z)^4\times \Z/4\Z$ and  
$J_4(\Q)_{\tors}\cong (\Z/2\Z)^3\times \Z/6\Z$.
This completes the proof of Theorem~\ref{T:main}.
}

\bibliographystyle{alpha}
\bibliography{Bibliography}

\end{document}